\documentclass[12pt,reqno]{amsart}
\usepackage{xcolor}
\usepackage{amsthm,amsmath,amssymb,latexsym,soul,cite,mathrsfs}
\usepackage{enumitem}
\usepackage{dsfont}
\usepackage{mathtools}
\usepackage{color,enumitem,graphicx}
\usepackage[colorlinks=true,urlcolor=blue,
citecolor=red,linkcolor=blue,linktocpage,pdfpagelabels,
bookmarksnumbered,bookmarksopen]{hyperref}
\usepackage[english]{babel}

\usepackage[left=2.7cm,right=2.7cm,top=2.9cm,bottom=2.9cm]{geometry}

\usepackage[hyperpageref]{backref}
\usepackage{dsfont}

\usepackage{cleveref}
\crefname{section}{Sect.}{Sect. }
\Crefname{section}{\textsection}{\textsection}
\numberwithin{equation}{section}

\theoremstyle{plain}
\newtheorem{theorem}{Theorem}[section]
\newtheorem{theoremletter}{Theorem}

\newtheorem{lemma}[theorem]{Lemma}
\newtheorem{lemmaletter}[theoremletter]{Lemma}

\theoremstyle{remark}

\newtheorem{remark}[theorem]{Remark}
\newtheorem{Claim}{Claim}
\theoremstyle{definition}

\newtheorem{definition}[theorem]{Definition}

\makeatletter
\newcommand{\leqnomode}{\tagsleft@true}
\newcommand{\reqnomode}{\tagsleft@false}
\makeatother

\title[Quasilinear Lane-Emden type systems]
{Quasilinear Lane-Emden type systems with sub-natural growth terms}

\author[E. da Silva]{Estevan Luiz da Silva} 
\author[J. M.\ do \'O]{Jo\~ao Marcos do \'O*}

\address[E. da Silva]{Department of Mathematics,
	Federal University of Pernambuco
	\newline\indent
	50740-540, Recife - PE, Brazil}
\email{\href{mailto:estevan.luiz@ufpe.br}{estevan.luiz@ufpe.br}}

\address[J.M. do \'O]{Department of Mathematics,
Federal University of Para\'{\i}ba
\newline\indent
58051-900, Jo\~ao Pessoa-PB, Brazil}
\email{\href{mailto:jmbo@pq.cnpq.br}{jmbo@pq.cnpq.br}}

\thanks{* Corresponding author.}
 

\begin{document}

%
%

\begin{abstract} 
Global pointwise estimates are obtained for quasilinear Lane-Emden-type systems involving measures in the ``sublinear growth'' rate. We give necessary and sufficient conditions for existence expressed in terms of Wolff's potential. Our approach is based on recent advances due to Kilpel\"{a}inen and Mal\'{y} in the potential theory. This method enables us to treat several problems, such as equations involving general quasilinear operators and fractional Laplacian.
\end{abstract}

\maketitle




%
%

\section{Introduction}\label{intro}
In this work, we give necessary and sufficient conditions for the existence of distributional solutions for systems of the Lane-Emden type, including the following model problem for quasilinear elliptic equations:	
\begin{equation}\label{plaplaciansystem}\tag{$S$}
     \left\{
\begin{aligned}
& -\Delta_p u= \sigma\, v^{q_1}, & u>0 \quad \mbox{in}\quad \mathds{R}^n,&\\ 
& -\Delta_p v= \sigma\, u^{q_2}, & v>0 \quad \mbox{in}\quad \mathds{R}^n,&\\
& \varliminf_{|x|\to \infty} u(x)=0, & \varliminf_{|x|\to \infty} v(x)=0, &
\end{aligned}
\right.
 \end{equation}
where $n\geq 3$, $\Delta_p f=\mathrm{div}(|\nabla f|^{p-2}\nabla f)$ is the $p$-Laplacian, $1<  p<\infty$ and $\sigma$ belongs to  $M^+(\mathds{R}^n)$, the set of all nonnegative Radon measures on $\mathds{R}^n$. We are interested in studying Syst.~\eqref{plaplaciansystem} for the sub-natural growth, that is, the case $ \, 0<q_i<p-1$ for $ i=1,\, 2.$

For a single semilinear elliptic equation that corresponds to  $p=2$, 
the sublinear problem was studied in \cite{MR1141779}, where the authors present a necessary and sufficient condition for the existence and uniqueness of \textit{bounded} solutions.
Precisely, they considered the equation $-\Delta u = \sigma(x) u^q$ in $\mathds{R}^n (n \geq 3)$  with $0<q<1, \; 0 \not \equiv  \sigma \geq 0$ and $\sigma \in L^{\infty}_{\mathrm{loc}}(\mathds{R})$. Using the Green function of the Laplacian in balls of $\mathds{R}^n$, they proved global pointwise estimates for bounded solutions of the semilinear elliptic equation  of the form 
\begin{equation}\label{BK estimate}
    c^{-1}\left(\mathbf{I}_2\sigma\right)^{\frac{1}{1-q}}\leq u\leq c\,\mathbf{I}_2\sigma,
\end{equation}
where $c>0$ is a constant independent of $u$ and 
\begin{equation*}
\mathbf{I}_2\sigma(x)=\int_{\mathds{R}^n}\frac{\mathrm{d}\sigma(y)}{|x-y|^{n-2}}
\end{equation*}
is the Newtonian potential of $\sigma$. In \cite{MR3556326} were obtained pointwise estimates of {B}rezis-{K}amin type for solutions of quasilinear elliptic equations involving sub-natural growth. See also \cite{MR3311903, MR3567503} for some related results.

Our main goal is to obtain necessary and sufficient conditions on $\sigma$ to prove the existence of distributional solutions to Syst.~\eqref{plaplaciansystem}  satisfying sharp global pointwise estimates of Brezis–Kamin type in terms of the Wolff potential. In our approach, using some ideas developed in \cite{MR3567503, MR3556326, MR2305115}; first, we transform Syst.~\eqref{plaplaciansystem} into a system of integral equations. Then, under suitable assumptions on $\sigma$, we can prove the existence of solutions for this associated system of integral equations based on the method of sub- and super-solutions.   Because of this result,  by using the method of successive approximations, we show the existence of distributional solutions (possibly unbounded) to Syst.~\eqref{plaplaciansystem} jointly with sharp global pointwise bounds of solutions.  Our results are new even for systems of nonlinear elliptic equations involving the classical Laplace operator, corresponding to the case $p=2$; or for nonnegative $\sigma\in L_{\mathrm{loc}}^1(\mathds{R}^n)$, where here $\mathrm{d}\sigma=\sigma\,\mathrm{d}x$.

\subsection{Systems of Wolff potentials}
To study systems of quasilinear elliptic equations like \eqref{plaplaciansystem}, we introduce some elements of nonlinear potential theory that will be used throughout the paper. 
To be specific, we consider the Wolff potential $\mathbf{W}_{\alpha,p}\sigma$ defined by
\begin{equation}\label{wolff potential}
    \mathbf{W}_{\alpha,p}\sigma(x)=\int_0^{\infty} \left(\frac{\sigma(B(x,t))}{t^{n-\alpha p}}\right)^{\frac{1}{p-1}}\frac{\mathrm{d} t}{t}, \quad x\in\mathds{R}^n,
\end{equation}
 where $\sigma\in {M}^+(\mathds{R}^n), \; 1<p<\infty, \; 0<\alpha<{n}/{p}$
 and $B(x,t)$ is the open ball of radius $t>0$ centered at $x$.  
 Observe that $\mathbf{W}_{\alpha,p}\sigma$ is always positive for $\sigma\not\equiv 0$ which may be $\infty$.
 
To solve Syst.~\eqref{plaplaciansystem}, first 
we study in a general framework the following system of integral equations involving the Wolff potentials,
\begin{equation}\label{sistemawolff}\tag{$SI$}
    \left\{
\begin{aligned}
& u= \mathbf{W}_{\alpha,p}\left(v^{q_1}\mathrm{d} \sigma\right), \quad \mathrm{d}\sigma \mbox{-a.e in } \mathds{R}^n,\\
& v= \mathbf{W}_{\alpha,p}\left(u^{q_2}\mathrm{d} \sigma\right), \quad \mathrm{d}\sigma \mbox{-a.e in } \mathds{R}^n,
\end{aligned}
\right.
\end{equation}
where $u\in L_{\mathrm{loc}}^{q_2}(\mathds{R}^n,\, \mathrm{d} \sigma)$ and $v\in L_{\mathrm{loc}}^{q_1}(\mathds{R}^n,\, \mathrm{d}\sigma)$ are nonnegative functions. We prove the existence of solutions to Syst.~\eqref{sistemawolff} by using the sub and super solutions method. 
Indeed, solutions of \eqref{sistemawolff} work like an upper barrier, which will allow us to control the successive approximations.

In the context of quasilinear problems, Wolff potentials with $\alpha=1$ appeared in the celebrated works of T. Kilpel\"{a}inen and J. Mal\'{y} \cite{MR1264000, MR1205885}. A global version of one of their main theorems states that if $u$ is a solution understood in the potential theoretic sense of the equation
\begin{equation}\label{kilpelainenmalequation}
     \left\{
\begin{aligned}
& -\Delta_p u= \sigma, \quad u>0 \quad \mbox{in}\quad \mathds{R}^n,\\
& \quad \varliminf_{|x|\to \infty}u(x)=0,
\end{aligned}
\right.
\end{equation}
then there exists a constant $K\geq 1$ which depends only on $p$ and $n$ such that
\begin{equation}\label{potential estimate}
    K^{-1}\mathbf{W}_{1,p} \sigma(x) \leq u(x) \leq K\,\mathbf{W}_{1,p}\sigma(x), \quad x\in \mathds{R}^n,
\end{equation}
see definition in Sect. \ref{ preliminaries}. 
    Moreover, $u$ exists if and only if $\mathbf{W}_{1,p}\sigma \not\equiv \infty$, or equivalently, 
\begin{equation}\label{potencial finito}
    \int_1^{\infty} \left(\frac{\sigma(B(0,t))}{t^{n- p}}\right)^{\frac{1}{p-1}}\, \frac{\mathrm{d} t}{t}<\infty.
\end{equation}
For an overview of Wolff potentials and their applications in Analysis and PDE, see \cite{MR1411441, MR727526, MR3174278, MR2777530, MR2305115} and references therein.


\subsection{Main results} For the next two theorems, we assume that $\mathbf{W}_{\alpha,p}\sigma(x)\not\equiv\infty$ for almost every $x\in \mathds{R}^n$, or equivalently, by \cite[Corollary~3.2]{MR3567503}
    \begin{equation}\label{general wolff finite}
        \int_1^\infty\left(\frac{\sigma(B(0,t))}{t^{n-\alpha p}}\right)^{\frac{1}{p-1}}\frac{\mathrm{d} t}{t}<\infty.
    \end{equation}
This means that $\mathbf{W}_{\alpha,p}\sigma(x)<\infty$ for almost every $x\in \mathds{R}^n$ if and only if $\mathbf{W}_{\alpha,p}\sigma(x_0)<\infty$ for some $x_0\in \mathds{R}^n$.

To state our first result in a precise way, let us recall the notion of $(\alpha, p)$-capacity of a compact set $E\subset\mathds{R}^n,$
\begin{equation}\label{capa}
    \mathrm{cap}_{\alpha,p}(E)=\inf \{\|f\|_{L^{p}}^p: 
    f \in L^{p}(\mathds{R}^n), \;  f \geq 0, \; \mathbf{I}_{\alpha}f\geq 1\ \mbox{on}\ E\},
\end{equation}
where $\mathbf{I}_{\alpha}\sigma$ is the Riesz potential of order $\alpha$ defined for $0<\alpha< n$ by
\begin{equation}\label{RieszPotential}
    \mathbf{I}_{\alpha}\sigma(x)=\int_{\mathds{R}^n}\frac{\mathrm{d}\sigma(y)}{|x-y|^{n-\alpha}}, \quad x\in \mathds{R}^n.
\end{equation}
We will show in Lemma \ref{Lemma sigma abs cont} below that if there exists a nontrivial solution to Syst. \eqref{sistemawolff}, then $\sigma$  must be absolutely continuous with respect to $\mathrm{cap}_{\alpha,p}$. Indeed, for our main results, we will impose
that $\sigma$ satisfies the following strongest condition
\begin{equation}\label{sigma abscont alfa p capacidade}
           \sigma(E)\leq C_{\sigma}\,\mathrm{cap}_{\alpha,p}(E) \quad \mbox{for all compact sets } E\subset \mathds{R}^n.
\end{equation}

\begin{theorem}\label{existencewollfsystem}
Let $1<p<\infty$, $0<q_i<p-1$, $i=1,2, \; 0<\alpha<{n}/{p}$ and $\sigma\in  {M}^+(\mathds{R}^n)$ satisfying \eqref{general wolff finite} and \eqref{sigma abscont alfa p capacidade}. Then there exists a solution $(u,v)$ to Syst. \eqref{sistemawolff} such that 
\begin{equation}\label{estimate upper and lower}
    \begin{aligned}
    & c^{-1}\left(\mathbf{W}_{\alpha,p}\sigma\right)^{\frac{(p-1)(p-1+q_1)}{(p-1)^2-q_1q_2}}\leq u\leq c\left(\mathbf{W}_{\alpha,p}\sigma + \left(\mathbf{W}_{\alpha,p}\sigma\right)^{\frac{(p-1)(p-1+q_1)}{(p-1)^2-q_1q_2}} \right),\\
   & c^{-1}\left(\mathbf{W}_{\alpha,p}\sigma\right)^{\frac{(p-1)(p-1+q_2)}{(p-1)^2-q_1q_2}}\leq v\leq c\left(\mathbf{W}_{\alpha,p}\sigma + \left(\mathbf{W}_{\alpha,p}\sigma\right)^{\frac{(p-1)(p-1+q_2)}{(p-1)^2-q_1q_2}} \right),
    \end{aligned}
\end{equation}
where $c=c(n,p,q_1,q_2,\alpha,C_{\sigma})>0$. Furthermore, $u,v \in L_{\mathrm{loc}}^{s}(\mathds{R}^n,\, \mathrm{d} \sigma)$, for every $s>0$.
\end{theorem}

\begin{remark}
Based on the assumption \eqref{general wolff finite}, we show that all nontrivial solutions to Syst~\eqref{sistemawolff} satisfy the lower bounds in \eqref{estimate upper and lower}. 
For the upper bounds in \eqref{estimate upper and lower}, we also use hypothesis \eqref{sigma abscont alfa p capacidade}, which will be decisive in building our argument. 
\end{remark}

For the next result, we assume the following condition on $\sigma$,
\begin{equation}\label{weakercontidion}
\begin{aligned}
&    \mathbf{W}_{\alpha,p}\left(\left(\mathbf{W}_{\alpha,p}\sigma\right)^{\frac{(p-1)(p-1+q_1)q_2}{(p-1)^2-q_1q_2}} \mathrm{d}\sigma\right)\leq \lambda \left(\mathbf{W}_{\alpha,p}\sigma + \left(\mathbf{W}_{\alpha,p}\sigma\right)^{\frac{(p-1)(p-1+q_1)}{(p-1)^2-q_1q_2}} \right), \\
& \mathbf{W}_{\alpha,p}\left(\left(\mathbf{W}_{\alpha,p}\sigma\right)^{\frac{(p-1)(p-1+q_2)q_1}{(p-1)^2-q_1q_2}}\mathrm{d} \sigma\right)\leq \lambda \left(\mathbf{W}_{\alpha,p}\sigma + \left(\mathbf{W}_{\alpha,p}\sigma\right)^{\frac{(p-1)(p-1+q_2)}{(p-1)^2-q_1q_2}} \right),
\end{aligned}
 \end{equation}
where the right-hand sides are finite almost everywhere in $\mathds{R}^n$, and $\lambda$ is a positive parameter. In general,  condition \eqref{weakercontidion} is weaker than capacity condition \eqref{sigma abscont alfa p capacidade} (see \cite{MR3556326} for more details). However, we can still construct solutions for Syst. \eqref{sistemawolff} satisfying \eqref{estimate upper and lower} and show that \eqref{weakercontidion} is necessary for the existence of such solutions.

\begin{theorem}\label{existencewolffsystemlocweaker}
Let $1<p<\infty$, $0<q_i<p-1$, $i=1, 2$, and $0<\alpha<{n}/{p}$. If $\sigma\in  {M}^+(\mathds{R}^n)$ satisfies \eqref{general wolff finite} and \eqref{weakercontidion}, then there exists a solution $(u,v)$ to 
Syst.~\eqref{sistemawolff} such that \eqref{estimate upper and lower} holds with a positive constant $c=c(n,p,\alpha,q_1,q_2,\lambda)$. Conversely, suppose that there exists a nontrivial solution $(u,v)$ to Syst.~\eqref{sistemawolff} satisfying \eqref{estimate upper and lower}. Then \eqref{weakercontidion} holds with $\lambda$ depending only on $p,q_1,q_2$ and $c$.
\end{theorem}

\begin{remark}
For the case  $q_1=q_2$,  one can see that the solution $(u,v)$ obtained in Theorem~\ref{existencewollfsystem} is such that $u=v$, and satisfies 
 \begin{equation}\label{estimative for doubleequation}
c^{-1}\big(\mathbf{W}_{\alpha,p}\sigma\big)^{\frac{p-1}{p-1-q}}\leq u\leq c\left(\mathbf{W}_{\alpha,p}\sigma + \left(\mathbf{W}_{\alpha,p}\sigma\right)^{\frac{p-1}{p-1-q}} \right),
  \end{equation}
   where $q:=q_1$. Thus, \cite[Theorem~3.2]{MR3556326}  is a corollary of Theorem \ref{existencewollfsystem} for $q_1=q_2$. Also, \cite[Theorem~3.3]{MR3556326} is a corollary of Theorem~\ref{existencewolffsystemlocweaker}, since in the case $q_1=q_2$ condition \eqref{weakercontidion} is written as follows 
   \begin{equation*}
\mathbf{W}_{\alpha,p}\left(\left(\mathbf{W}_{\alpha,p}\sigma\right)^{\frac{q(p-1)}{p-1-q}} \mathrm{d}\sigma\right)\leq \lambda \left(\mathbf{W}_{\alpha,p}\sigma + \left(\mathbf{W}_{\alpha,p}\sigma\right)^{\frac{p-1}{p-1-q}} \right)<\infty.
   \end{equation*}
\end{remark}
\subsection{Application}
We will apply the preceding theorems to obtain solutions to Syst.~\eqref{plaplaciansystem}.
 For that, we recall the $p$-capacity for compact subsets $E$ of $\mathbb{R}^n,$
 \begin{equation*}
     \mathrm{cap}_{p}(E)=\inf\left\{\|\nabla \varphi\|_{L^p}^{p}: \varphi\in C_c^{\infty}(\mathds{R}^n), \varphi\geq 1 \,\mbox{on}\, E\right\}.
 \end{equation*}
 We remark that  $\mathrm{cap}_{1,p}(E)\approx \mathrm{cap}_p(E)$ for all compact sets $E$, where the notation $A \approx B$ means that there exists a positive constant $c$ such that  $c^{-1}B \leq A \leq c\, B$ (see \cite{MR1411441}).
 
 We assume that $\sigma$ satisfies the capacity condition
\begin{equation}\label{sigma abs cap_p}
    \sigma(E)\leq C_{\sigma}\,\mathrm{cap}_p(E) \quad \mbox{for all compact sets }E\subset\mathds{R}^n.
\end{equation}
We recall that  $W_{\mathrm{loc}}^{1,p}(\mathds{R}^n)$ is the space of all functions $u \in L_{\mathrm{loc}}^p(\mathds{R}^n)$ which admit weak derivatives $\partial_i u \in L_{\mathrm{loc}}^p(\mathds{R}^n)$ for $i=1,\ldots,n$.

Under assumption \eqref{sigma abs cap_p}, we prove the existence of a minimal positive 
\textit{$p$-superharmonic solution} to Syst.~\eqref{plaplaciansystem}, more precisely, a pair 
$(u,v)\in W_{\mathrm{loc}}^{1,p}(\mathds{R}^n)\times W_{\mathrm{loc}}^{1,p}(\mathds{R}^n)$ such that $u$, $v$ are positive and $p$-superharmonic functions in $\mathds{R}^n$, and satisfies Syst.~\eqref{plaplaciansystem} in the  distributional sense. Here a minimal solution $(u,v)$ to Syst.~\eqref{plaplaciansystem} means that for any solution 
$(\tilde{u},\tilde{v})$ to \eqref{plaplaciansystem}, we have $\tilde{u}\geq u$ and $\tilde{v}\geq v$ almost everywhere in $\mathds{R}^n$. 
See Sect. \ref{ preliminaries} for the definition of solutions to Syst.~\eqref{plaplaciansystem} and $p$-superharmonic functions.
\begin{theorem}\label{solutionplaplacian general}
Let $1<p<n$ and $0<q_i<p-1$ for $i=1,2$.
Suppose $\sigma\in {M}^+(\mathds{R}^n)$ satisfies \eqref{potencial finito} and \eqref{sigma abs cap_p}. Then there exists a minimal $p$-superharmonic positive solution $(u,v)\in W_{\mathrm{loc}}^{1,p}(\mathds{R}^n)\times W_{\mathrm{loc}}^{1,p}(\mathds{R}^n)$ to Syst.~\eqref{plaplaciansystem} such that almost everywhere we have
\begin{equation}\label{estimative upper and lower p-laplacian}
    \begin{aligned}
    & c^{-1}\left(\mathbf{W}_{1,p}\sigma\right)^{\frac{(p-1)(p-1+q_1)}{(p-1)^2-q_1q_2}}\leq u\leq c\left(\mathbf{W}_{1,p}\sigma + \left(\mathbf{W}_{1,p}\sigma\right)^{\frac{(p-1)(p-1+q_1)}{(p-1)^2-q_1q_2}} \right),\\
   & c^{-1}\left(\mathbf{W}_{1,p}\sigma\right)^{\frac{(p-1)(p-1+q_2)}{(p-1)^2-q_1q_2}}\leq v\leq c\left(\mathbf{W}_{1,p}\sigma + \left(\mathbf{W}_{1,p}\sigma\right)^{\frac{(p-1)(p-1+q_2)}{(p-1)^2-q_1q_2}} \right),
    \end{aligned}
\end{equation}
where $c=c(n,p,q_1,q_2,C_{\sigma})>0$. 
If $p\geq n$, there are no nontrivial solutions to Syst.~\eqref{plaplaciansystem} in $\mathds{R}^n$.
\end{theorem}

\begin{remark}
  In view of Theorem~\ref{constante Maly} and Lemma~\ref{estimativainferior}, the lower estimates in \eqref{estimative upper and lower p-laplacian} holds for any distributional solution $(u,v)$ to \eqref{plaplaciansystem}. Nevertheless, the upper estimates in \eqref{estimative upper and lower p-laplacian} is obtained only for the minimal solution.
\end{remark}
In the next theorem, we give a necessary and sufficient condition for the existence of a distributional solution to Syst.~\eqref{plaplaciansystem}  satisfying \eqref{estimative upper and lower p-laplacian}.
\begin{theorem}\label{solutionplaplacian general weaker}
Let $\sigma\in {M}^+(\mathds{R}^n), \; 1<p<n$ and $0<q_i<p-1$ for $i=1,2$. 
Then there exists a $p$-superharmonic positive solution $(u,v)$ to  Syst.~\eqref{plaplaciansystem} satisfying \eqref{estimative upper and lower p-laplacian} if and only if there exists $\lambda>0$ such that almost everywhere we have
\begin{equation}\label{condition necessary weaker}
    \begin{aligned}
  & \mathbf{W}_{1,p}\left((\mathbf{W}_{1,p}\sigma)^{\frac{q_2(p-1)(p-1+q_1)}{(p-1)^2-q_1q_2}} \right)\leq \lambda\left(\mathbf{W}_{1,p}\sigma + \left(\mathbf{W}_{1,p}\sigma\right)^{\frac{(p-1)(p-1+q_1)}{(p-1)^2-q_1q_2}} \right) <\infty, 
  \\
   & \mathbf{W}_{1,p}\left((\mathbf{W}_{1,p}\sigma)^{\frac{q_1(p-1)(p-1+q_2)}{(p-1)^2-q_1q_2}} \right)\leq \lambda\left(\mathbf{W}_{1,p}\sigma + \left(\mathbf{W}_{1,p}\sigma\right)^{\frac{(p-1)(p-1+q_2)}{(p-1)^2-q_1q_2}} \right) <\infty .
    \end{aligned}
\end{equation}
\end{theorem}

We also consider solutions to systems of equations involving the fractional Laplacian of the form	
\begin{equation}\label{fractlaplaciamsystem}
     \left\{
\begin{aligned}
& (-\Delta)^{\alpha} u= \sigma\, v^{q_1}, \quad v>0 \quad \mbox{in}\quad \mathds{R}^n,\\ 
& (-\Delta)^{\alpha} v= \sigma\, u^{q_2}, \quad u>0 \quad \mbox{in}\quad \mathds{R}^n,\\
& \varliminf_{|x|\to \infty}u(x)=0, \quad \varliminf_{|x|\to \infty}v(x)=0.
\end{aligned}
\right.
 \end{equation}
 with $0<\alpha<n/2$ and $q_1,~ q_2\in(0,1)$. A solution $(u,v)$ to \eqref{fractlaplaciamsystem} is understood in the sense
 \begin{equation*}
 \left\{
\begin{aligned}
& u(x)=\frac{1}{c(n,\alpha)}\mathbf{I}_{2\alpha}(v^{q_1}\mathrm{d} \sigma)(x), \quad x\in \mathds{R}^n,\\ 
& v(x)=\frac{1}{c(n,\alpha)}\mathbf{I}_{2\alpha}(u^{q_2}\mathrm{d} \sigma)(x), \quad x\in \mathds{R}^n,
\end{aligned}
\right.
 \end{equation*}
where $\mathbf{I}_{2\alpha}$ is the Riesz Potential defined in \eqref{RieszPotential} and
$c(n,\alpha)$ is a normalization constant given by
\begin{equation*}
    c(n,\alpha)=\frac{2^{\alpha}\pi^{\frac{n}{2}}\Gamma(\alpha)}{\Gamma(\frac{n-2\alpha}{2})}
\end{equation*}
(see  for instance \cite{MR4043885}).
From now on, the normalization constant will be dropped for the sake of convenience.
Thus, Syst.~\eqref{fractlaplaciamsystem} is equivalent (up to a constant) to the system integral \eqref{sistemawolff}, since 
\begin{equation*}
    \mathbf{I}_{2\alpha}\sigma (x)=\int_{\mathds{R}^n} \frac{\mathrm{d}\sigma(y)}{|x-y|^{n-2\alpha}}=(n-2\alpha)\int_{0}^\infty \frac{\sigma(B(x,t))}{t^{n-2\alpha}}\frac{\mathrm{d} t}{t}=(n-2\alpha)\mathbf{W}_{\alpha,2}\sigma(x).
\end{equation*}

Therefore, the following theorem is a special case of Theorem \ref{existencewollfsystem} with $p=2$.

\begin{theorem}\label{thm frac system}
Let $0<2\alpha<{n}$, $0<q_i<1$, $i=1,2$. Suppose $\sigma\in {M}^+(\mathds{R}^n)$ satisfies both \eqref{general wolff finite} and \eqref{sigma abscont alfa p capacidade} with $p=2$. Then Syst.~\eqref{fractlaplaciamsystem} admits a solution such that
\begin{equation}\label{estimate riesz}
     \begin{aligned}
    & c^{-1}\left(\mathbf{I}_{2\alpha}\sigma\right)^{\frac{1+q_1}{1-q_1q_2}}\leq u\leq c\left(\mathbf{I}_{2\alpha}\sigma + \left(\mathbf{I}_{2\alpha}\sigma\right)^{\frac{1+q_1}{1-q_1q_2}} \right),\\
   & c^{-1}\left(\mathbf{I}_{2\alpha}\sigma\right)^{\frac{1+q_2}{1-q_1q_2}}\leq v\leq c\left(\mathbf{I}_{2\alpha}\sigma + \left(\mathbf{I}_{2\alpha}\sigma\right)^{\frac{1+q_2}{1-q_1q_2}} \right),
    \end{aligned}
\end{equation}
where $c=c(n,q_1,q_2,\alpha,C_{\sigma})>0$. Furthermore, $u, v \in L_{\mathrm{loc}}^{s}(\mathds{R}^n,\, \mathrm{d} \sigma)$, for every $s>0$.
\end{theorem}

\begin{remark}
    Suppose $\mathbf{I}_{2\alpha}\sigma \in L^{\infty}(\mathds{R}^n)$ as in \cite{MR1141779}, 
    then assumption  \eqref{sigma abscont alfa p capacidade} holds for $p=2$. 
    Indeed, in view of \cite[Theorem~1.11]{MR1747901}, there exists a constant $C_0>0$ which depends only on $n$ and $\alpha$, such that, for all compact sets $E\subset \mathds{R}^n$,
\begin{equation*}
\begin{aligned}
   \sigma(E) & = \int_E\mathrm{d}\sigma = \int_E \mathbf{W}_{\alpha,2}\sigma \,\frac{\mathrm{d}\sigma}{\mathbf{W}_{\alpha,2}\sigma} =\int_E \mathbf{I}_{2\alpha}\sigma\,\frac{\mathrm{d}\sigma}{\mathbf{I}_{2\alpha}\sigma} \\
    & \leq \|\mathbf{I}_{2\alpha}\sigma\|_{L^{\infty}}\int_E \frac{\mathrm{d}\sigma}{\mathbf{I}_{2\alpha}\sigma}\\
    & \leq \|\mathbf{I}_{2\alpha}\sigma\|_{L^{\infty}}\, C_0 \,\,\mathrm{cap}_{\alpha, 2}(E),
\end{aligned}    
\end{equation*}
and this shows condition \eqref{sigma abscont alfa p capacidade} for $p=2$. Thus, under assumption $\mathbf{I}_{2\alpha}\sigma \in L^{\infty}(\mathds{R}^n)$, 
from Theorem~\ref{thm frac system}, it follows that Syst.~\eqref{fractlaplaciamsystem} possess a \textit{bounded} nontrivial solution $(u,v)$ satisfying the so-called Brezis–Kamin estimate \eqref{estimate riesz}. In particular, Theorem~\ref{thm frac system}, with $\alpha=1$ and $q_1=q_2$, recover \cite[Theorem~1]{MR1141779} together with the estimate \eqref{BK estimate}, provided $\mathbf{I}_{2}\sigma \in L^{\infty}(\mathds{R}^n)$.
\end{remark}

\subsection{Related results}
Recently, studies of systems involving the $p$-Laplace operator with measure-valued right-hand side have been done. In this respect, in \cite{MR3779689}, T. Kuusi and G. Mingione used the notion SOLA (Solution Obtained as Limits of Approximations) to propose a vectorial version of \cite{MR1205885}. They use a different approach from ours to establish local upper pointwise potential estimates for $W_{\mathrm{loc}}^{1,p-1}$-vectorial solutions in terms of Wolff's potential if $p>2-1/n$,  and its gradient in terms of Riesz's potential if $p>2$, under the standard assumptions $|\sigma|(\mathds{R}^n)<\infty$ (see also \cite{MR3485149}).  
We emphasize that the present work brings a global and more accurate estimate depending only on Wolff's potential of $\sigma$ of some distributional solutions to Syst.~\eqref{plaplaciansystem} than the one in \cite{MR3779689}. For more results of distributional solutions to the $p$-Laplacian system with the measure-valued right-hand side, see, for instance, \cite{MR1450953}.  G. Dolzmann, N. Hungerb\"{u}hler and S. M\"{u}ller in \cite{MR1450953} proved the existence of distributional solutions for $p$-harmonic functions and established the Lorentz space estimates for such solutions.
I. Chlebicka, Y. Youn, and A. Zatorska-Goldstein \cite{MR4530311} applied the approach introduced in \cite{MR3779689} to study solutions to measure data elliptic systems involving operators of the divergence form with Orlicz growth. They provided pointwise estimates for the solutions expressed in terms of a nonlinear potential of generalized Wolff type.

\subsection{Organization of the paper} 

In Sect.~\ref{ preliminaries}, we present auxiliary results from nonlinear potential theory. 
In Sect.~\ref{section3}, we produce a framework to solve Syst.~\eqref{plaplaciansystem} by proving Theorem~\ref{existencewollfsystem} and Theorem~\ref{existencewolffsystemlocweaker}. In Sect.~\ref{section4}, we give the proofs of Theorems~\ref{solutionplaplacian general}, Theorem~\ref{solutionplaplacian general weaker}, and Theorem~\ref{thm frac system}. In Sect.~\ref{final comments}, we finish with some consequences of another type of problem.

\subsection{Notations and definitions}  
\begin{itemize}
    \item  $\Omega$ is a domain in $\mathds{R}^n$.
    \item As usual, we use the letters $c$, $\tilde{c}$, $C$, and $\tilde{C}$,  with or without subscripts, to denote different constants.
    \item $\chi_E:=$ the characteristic function of a set $E$.
    \item $ {M}^+(\Omega):=$ the set of all nonnegative Radon measures $\sigma$ defined  on $\Omega$. 
    \item Often, we use the Greek letters $\mu$ and $\omega$ to denote Radon measures.
    \item  We just denote by $ \mathrm{d}\sigma \mbox{-a.e in } \mathds{R}^n,$ when a property holds almost everywhere in $\mathds{R}^n$ in the sense of the measure $\sigma$.
    \item $C(\Omega):=$ the set of all continuous functions in on $\Omega$.
    \item $C_c^{\infty}(\Omega):=$ the set of all infinitely differentiable functions with compact support in $\Omega$. 
    \item $L_{\mathrm{loc}}^s(\Omega, \mathrm{d}\mu):=$ the local $L^s$ space with respect to $\mu\in  {M}^+{(\Omega)}$, $s>0$. If $\mu$ is the Lebesgue measure, we write $L_{\mathrm{loc}}^s(\Omega)$. 
    \item $W_{\mathrm{loc}}^{1,p}(\mathds{R}^n)$ is the space of all functions $u \in L_{\mathrm{loc}}^p(\mathds{R}^n)$ which admit weak derivatives $\partial_i u \in L_{\mathrm{loc}}^p(\mathds{R}^n)$ for $i=1,\ldots,n$.
    \item $ W_{\mathrm{loc}}^{-1,p'}(\Omega) $ is the dual of the Sobolev space  $W_{\mathrm{loc}}^{1,p}(\mathds{R}^n)$, where $p'=p/(p-1)$.
   \item We denote by $\sigma(E)=\int_E \mathrm{d}\sigma$ the measure of any $\sigma$-measurable subset $E$ of $\Omega$.
\end{itemize}

   \section{Preliminaries}\label{ preliminaries} 
We begin this section by recalling some basic definitions and results for easy reference.  For details see  \cite{MR1205885, MR1264000, MR2305115}.
\begin{definition}\label{maracatu} Let $\Omega$ be a domain in $\mathds{R}^n$ and $1<p<\infty$. 
\begin{enumerate}
\item We define the $p$-Laplace operator for $w\in W_{\mathrm{loc}}^{1,p}(\Omega)$ in a distributional sense as follows 
\begin{equation*}
    \langle \Delta_pw,\varphi\rangle=\langle\mathrm{div}\left(|\nabla w|^{p-2}\nabla w\right),\varphi\rangle =-\int_\Omega |\nabla w|^{p-2}\nabla w\cdot\nabla \varphi\,\mathrm{d} x, \quad \forall \varphi\in C_c^{\infty}(\Omega).
\end{equation*}
\item We say that $w\in W_{\mathrm{loc}}^{1,p}(\Omega)$ is a distributional solution of the homogeneous $p$-laplacian equation ($p$-harmonic) if
$ \langle-\Delta_pw,\varphi\rangle=0, \;\forall  \varphi\in C_c^{\infty}(\Omega)$.
\item We define a  supersolution $w\in W_{\mathrm{loc}}^{1,p}(\Omega)$ in $\Omega$ if 
$\langle-\Delta_pw,\varphi\rangle \geq 0$
for all nonnegative $\varphi\in C_c^{\infty}(\Omega)$.
\end{enumerate}
\end{definition}

Next, we extend the notion of the distributional solutions for the equation  $-\Delta_p w=\mu$ where $w$ does not necessarily belong to $W_{\mathrm{loc}}^{1,p}(\Omega)$ and $\mu$ is in the dual space $ W_{\mathrm{loc}}^{-1,p'}(\Omega) $. Indeed, we will understand solutions in the
following potential-theoretic sense using $p$-superharmonic functions.

\begin{definition}\label{frevo}
A function $w:\Omega\to(-\infty,\infty) \cup \{\infty\}$ is $p$-superharmonic in $\Omega$ if 
\begin{enumerate}
    \item $w$ is lower semicontinuous,
    \item $w$ is not identically infinite in any component of $\Omega$,
    \item for each open subset  $D$ compactly contained in $\Omega$ and each $p$-harmonic function $h$ in $D$ such that $h \in  C(\overline{D})$  and $h\leq w$ in $\partial D$ implies $h\leq w$ in $D$.
\end{enumerate}
We denote by $\mathcal{S}_p(\Omega)$ for the class of all $p$-superharmonic functions in $\Omega$.
\end{definition}
   
For $w \in \mathcal{S}_p(\Omega)$ we define its truncation as follows 
\begin{equation*}
    T_k(w)=\min (k,\max(w,-k)), \quad \forall k >0. 
\end{equation*}
We mention that $ T_k(w) \in W_{\mathrm{loc}}^{1,p}(\Omega)$ although $w$ does not necessarily belong to $W_{\mathrm{loc}}^{1,p}(\Omega)$.
It is well known that $T_k(w)$ is a supersolution in $\Omega$, for all $k>0$, in the sense of Definition~\ref{maracatu} item (3). 

Following \cite{MR2305115}, we consider the generalized gradient for $w\in \mathcal{S}_p(\Omega)$ (see Definition~\ref{frevo}), as the following pointwiese limit 
\begin{equation*}
    Dw(x)=\lim_{k\to \infty}\nabla( T_k(w))(x) \quad \text{almost everywhere in }  \Omega.
\end{equation*}
If $p>2-1/n$, one checks that $Du$ is the distributional gradient of $u$ (see \cite[page~154]{MR2305115}). Using  \cite[Theorem~1.15]{MR1205885}, we see that for $w \in \mathcal{S}_p(\Omega)$  and $1\leq r<n/(n-1)$, we have $|Dw|^{p-1}\in L_{\mathrm{loc}}^{r}(\Omega)$. In particular,  $|Dw|^{p-2}Dw\in L_{\mathrm{loc}}^{r}(\Omega)$. 
 Because of this fact, we can define the $p$-Laplace operator in a distributional sense for $w\in \mathcal{S}_p(\Omega)$ as follows 
   \begin{equation*}
       \langle-\Delta_p w,\varphi\rangle=\int_\Omega |Dw|^{p-2}Dw\cdot \nabla \varphi\,\mathrm{d} x, \quad \forall \varphi\in C_c^{\infty}(\Omega).
   \end{equation*}
Therefore, by the Riesz Representation Theorem, there exists a unique measure $\mu=\mu[w]\in  {M}^+(\Omega)$ such that $-\Delta_p w=\mu[w]$. In the literature, $\mu[w]$  is called the Riesz measure of $w$.
   \begin{definition}\label{mangai}
       
For $ \sigma \in {M}^+(\Omega)$, we say that $w$ is  a solution in the potential-theoretic sense to the equation
\begin{equation*}
    -\Delta_p w= \sigma \quad \mbox{in}\quad \Omega
\end{equation*}
   if $w \in \mathcal{S}_p(\Omega)$ and $\mu[w]=\sigma$.
   \end{definition}
   In light of Definition~\ref{mangai}, if $\sigma\in  {M}^+(\Omega)$, then a pair $(u,v)$ is a solution (in the potential-theoretic sense) to the system
\begin{equation}\label{solutionsystemsense}
     \left\{
\begin{aligned}
& -\Delta_p u= \sigma\, v^{q_1} \quad \mbox{in}\quad \Omega,\\ 
& -\Delta_p v= \sigma\, u^{q_2} \quad \mbox{in}\quad \Omega
\end{aligned}
\right.
\end{equation}
whenever $u$ and $v$ are nonnegative functions and 
\begin{equation}
     \left\{
\begin{aligned}
& u\in \mathcal{S}_p(\Omega)\cap L_{\mathrm{loc}}^{q_2}(\Omega, \mathrm{d} \sigma), \\
& v\in \mathcal{S}_p(\Omega)\cap L_{\mathrm{loc}}^{q_1}(\Omega, \mathrm{d} \sigma), \\
& \mathrm{d} \mu[u]=v^{q_1}\mathrm{d} \sigma,\\ 
& \mathrm{d}\mu[v]=u^{q_2}\mathrm{d}\sigma.
\end{aligned}
\right.
\end{equation}

\begin{definition}\label{frevo1}
   Let $\sigma\in {M}^+(\mathds{R}^n)$. A pair $(u,v)$ is called a supersolution to Syst.~\eqref{plaplaciansystem} if $u$ and $v$ are nonnegative functions and
   \begin{equation}
     \left\{
\begin{aligned}
& u\in \mathcal{S}_p(\mathds{R}^n)\cap L_{\mathrm{loc}}^{q_2}(\mathds{R}^n, \mathrm{d} \sigma), \\
& v\in \mathcal{S}_p(\mathds{R}^n)\cap L_{\mathrm{loc}}^{q_1}(\mathds{R}^n, \mathrm{d} \sigma), \\
& \int_{\mathds{R}^n} |D u|^{p-2}Du\cdot \nabla \varphi\,\mathrm{d} x\geq \int_{\mathds{R}^n} v^{q_1}\varphi\,\mathrm{d}\sigma, \\ 
& \int_{\mathds{R}^n} |D v|^{p-2}Dv\cdot \nabla \varphi\,\mathrm{d} x\geq \int_{\mathds{R}^n} u^{q_2}\varphi\,\mathrm{d}\sigma, \quad \forall \varphi\in C_c^{\infty}(\mathds{R}^n),\, \varphi\geq 0.
\end{aligned}
\right.
\end{equation}
The notion of  \textit{subsolution} to Syst.~\eqref{plaplaciansystem} is defined similarly by replacing ``$\geq$'' by  ``$\leq$'' in Definition~\ref{frevo1}. 
\end{definition}

\begin{remark}\label{frevo2}
In view of Definitions \ref{frevo} and \ref{frevo1}, supersolutions (or solution) to Syst.~\eqref{plaplaciansystem} are supersolutions in $\mathds{R}^n$. Indeed, if $(u,v)$ is a supersolution to \eqref{plaplaciansystem} in the sense of Definition~\ref{frevo1} with $u,\, v \in W_{\mathrm{loc}}^{1,p}(\mathds{R}^n)$, then $u$ and $v$ are supersolutions in $\mathds{R}^n$ in the sense of Definition~\ref{frevo}.  \end{remark}

Next, we will employ some fundamental results of the potential theory of quasilinear elliptic
equations.  Let us state the next result, which will be used to prove that a pointwise limit of a sequence of $p$-superharmonic functions is, indeed, a $p$-superharmonic (see \cite[Lemma~7.3]{MR2305115}).

\begin{lemmaletter}
    \label{Lemma limit}
       Suppose that $\{w_j\}$ is a sequence of $p$-superharmonic functions in $\Omega$. If the sequence $\{w_j\}$ either inscreasing or converges uniformly on compact subsets in $\Omega$, then in each component of $\Omega$ the pointwise limit function $w=\lim_{j\to \infty}w_j$ is a $p$-superharmonic function unless $w\equiv \infty$.
\end{lemmaletter}

    The following weak continuity of the $p$-Laplacian result is due to 
    N. S. Trudinger and X.-J. Wang \cite{MR1890997} and will be used to prove the existence of $p$-superharmonic solutions to quasilinear equations.

\begin{theoremletter}
\label{weak continuity p-laplacian}
    Let $\{w_j\}$ be a sequence of nonnegative $p$-superharmonic functions in $\Omega$. Suppose that $\{w_j\}$ converge pointwise to $w$  where is finite almost everywhere and  $p$-superharmonic function in $\Omega$. Then $\mu[w_j]$ converges weakly to $\mu[w]$, that is,
    \begin{equation*}
        \lim_{j\to \infty}\int_{\Omega}\varphi\, \mathrm{d}\mu[w_j]=\int_{\Omega}\varphi\, \mathrm{d}\mu[w], \quad  \forall \varphi\in C_{c}^{\infty}(\Omega).
    \end{equation*}
\end{theoremletter}
    
Next, we estate a crucial result on pointwise estimates of nonnegative $p$-superharmonic functions in terms of Wolff's potential due to  T. Kilpel\"{a}inen and J. Mal\'y in \cite[Theorem~1.6]{MR1264000}.

\begin{theoremletter}
    \label{constante Maly}
    Let $w$ be a $p$-superharmonic function in $\mathds{R}^n$ with $\varliminf_{|x|\to \infty}w(x)=0$. If $1 < p<n$ and $\omega=\mu[w]$, that is, $-\Delta_p w=\omega$, then there exists a constant $K\geq 1$ depending only on $n$ and $p$ such that
        \begin{equation*}
            K^{-1}\mathbf{W}_{1,p}\omega(x)\leq w(x)\leq K\,\mathbf{W}_{1,p}\omega(x), \quad \forall  x\in\mathds{R}^n.
        \end{equation*}
\end{theoremletter}

We note that we can assume, without loss of generality,  that any $p$-superharmonic function $w$ in $\mathds{R}^n$ can be chosen to be quasicontinuous in $\mathds{R}^n$, that is, $w$ is a continuous function in $\mathds{R}^n$ except in a set of $p$-capacity zero, see more details in \cite[Chapter~7]{MR2305115}.

\section{Systems of Wolff potentials}\label{section3}

This section will consider supersolutions to the integral Syst.~\eqref{sistemawolff}. 
\begin{definition}
    We say that a pair of nonnegative functions $(v,u) \in L_{\mathrm{loc}}^{q_1}(\mathds{R}^n,\, \mathrm{d}\sigma)\times L_{\mathrm{loc}}^{q_2}(\mathds{R}^n,\, \mathrm{d} \sigma)$ is a \textit{supersolution} to \eqref{sistemawolff} if satisfies (pointwise)
\begin{equation}\label{sistemawolffsuper}
    \left\{ 
\begin{aligned}
& u(x) \geq  \mathbf{W}_{\alpha,p}\left(v^{q_1}\mathrm{d} \sigma\right)(x), \quad \mathrm{d}\sigma \mbox{-a.e in } \mathds{R}^n,\\
& v(x)\geq  \mathbf{W}_{\alpha,p}\left(u^{q_2}\mathrm{d} \sigma\right)(x),\quad \mathrm{d}\sigma \mbox{-a.e in } \mathds{R}^n.
\end{aligned}
\right.
\end{equation}
The notion of \textit{solution} or \textit{subsolution} to \eqref{sistemawolff} is defined similarly by replacing ``$\geq$'' by ``$=$'' or ``$\leq$'' in \eqref{sistemawolffsuper}, respectively. 
\end{definition}
We start showing that if there exists a nontrivial supersolution $(u,v)$ to \eqref{sistemawolff}, then $\sigma$ must be absolutely continuous with respect to the $(\alpha, p)$-capacity $\mathrm{cap}_{\alpha,p}$, see \eqref{capa}. In view of Theorem~\ref{constante Maly} and next lemma we see that if Syst.~\eqref{plaplaciansystem} has a nontrivial $p$-superharmonic supersolution, then $\sigma$ is absolutely continuous with respect to the $p$-capacity $\mathrm{cap}_p$, since $\mathrm{cap}_{p}(E) \approx \mathrm{cap}_{1,p}(E)$ for all compact sets $E$.
 \begin{lemma}\label{Lemma sigma abs cont}
    Let $1< p<\infty,\; 0<q_i<p - 1,\; i=1,\, 2,\;0<\alpha<{n}/{p}$, and $\sigma\in {M}^+(\mathds{R}^n)$. Suppose there is a nontrivial supersolution $(u,v)$ to \eqref{sistemawolff}. Then there exists a positive constant $C$ depending only on $n, p, q_1,q_2$ and  $\alpha$ such that  for every compact set $E\subset\mathds{R}^n$,
    \begin{equation}\label{sigma abs cont necessarily}
    \begin{aligned}
      \sigma(E)&\leq C\Big[\mathrm{cap}_{\alpha,p}(E)^{\frac{q_1}{p-1}}\Big(\int_E v^{q_1}\,\mathrm{d}\sigma\Big)^{\frac{p-1}{p-1-q_1}}+ \mathrm{cap}_{\alpha,p}(E)^{\frac{q_2}{p-1}}\Big(\int_E u^{q_2}\,\mathrm{d}\sigma\Big)^{\frac{p-1}{p-1-q_2}} \Big].
    \end{aligned}
      \end{equation}
 \end{lemma}
\begin{proof}
We first recall the following result \cite[Theorem~1.11]{MR1747901}: for any $\mu\in {M}^+(\mathds{R}^n)$ it holds 
\begin{equation*}
\int_E \frac{\mathrm{d} \mu}{\left(\mathbf{W}_{\alpha,p}\mu\right)^{p-1}}\leq C_0 \,\mathrm{cap}_{\alpha,p}(E), 
\end{equation*}
where $C_0=C_0(n,p,\alpha)$ is a positive constant.
Taking $\mathrm{d}\mu=v^{q_1}\mathrm{d}\sigma$, we obtain
\begin{equation*}
   \int_Ev^{q_1}u^{-(p-1)}\,\mathrm{d}\sigma \leq \int_{E}\frac{\mathrm{d}\mu}{(\mathbf{W}_{\alpha,p}\mu)^{p-1}}\leq C_0\,\mathrm{cap}_{\alpha,p}(E),
\end{equation*}
since $u\geq \mathbf{W}_{\alpha,p}\mu$. 
Thus,
\begin{align}\label{estimate3}
    \int_{E\cap\{v\geq u\}} v^{q_1-p+1}\,\mathrm{d} \sigma &\leq \int_{E\cap\{v\geq u\}} v^{q_1}u^{-(p-1)}\,\mathrm{d} \sigma\nonumber\\
    & \leq \int_E v^{q_1}u^{-(p-1)}\,\mathrm{d} \sigma\leq C_0\,\mathrm{cap}_{\alpha,p}(E).
\end{align}
Setting $\beta=q_1(p-1-q_1)/(p-1)$ and using the H\"{o}lder's inequality with exponents $r=(p-1)/q_1$ and $r'=(p-1)/(p-1-q_1)$, we deduce
\begin{align*}
     \sigma(E\cap\{v\geq u\})&= \int_{E\cap\{v\geq u\}}v^{-\beta}v^{\beta}\,\mathrm{d} \sigma\\
    &\leq \left(\int_{E\cap\{v\geq u\}} v^{-\beta r}\,\mathrm{d}\sigma\right)^{\frac{1}{r}}\left(\int_{E\cap\{v\geq u\}} v^{\beta r'}\,\mathrm{d}\sigma\right)^{\frac{1}{r'}} \\
    &=\left(\int_{E\cap\{v\geq u\}} v^{q_1-p+1}\,\mathrm{d}\sigma\right)^{\frac{q_1}{p-1}}\left(\int_{E\cap\{v\geq u\}} v^{q_1}\,\mathrm{d}\sigma\right)^{\frac{p-1-q_1}{p-1}},
\end{align*}
since $-\beta r=q_1-p+1$ and $\beta r'=q_1$.
By \eqref{estimate3},
\begin{equation*}
    \sigma(E\cap\{v\geq u\})\leq C_0^{\frac{q_1}{p-1}}\,\left(\mathrm{cap}_{\alpha,p}(E)\right)^{\frac{q_1}{p-1}}\left(\int_E v^{q_1}\,\mathrm{d}\sigma\right)^{\frac{p-1-q_1}{p-1}}.
\end{equation*}
Similarly, we also obtain
\begin{equation*}
    \int_{E\cap \{u\geq v\}} u^{q_2-p+1}\,\mathrm{d}\sigma \leq C_0\,\mathrm{cap}_{\alpha,p}(E),
\end{equation*}
and consequently
\begin{equation*}
    \sigma(E\cap\{u\geq v\})\leq C_0^{\frac{q_2}{p-1}}\,\left(\mathrm{cap}_{\alpha,p}(E)\right)^{\frac{q_2}{p-1}}\left(\int_{E} u^{q_2}\,\mathrm{d}\sigma\right)^{\frac{p-1-q_2}{p-1}}.
\end{equation*}
Since $\sigma(E)\leq \sigma(E\cap\{v\geq u\})+\sigma(E\cap\{u\geq v\})$, picking $C=\max\{C_0, 1\}$,  \eqref{sigma abs cont necessarily} follows.
\end{proof}

Notice that in Syst.~\eqref{sistemawolffsuper}, placing the second inequality  $v\geq \mathbf{W}_{\alpha,p}(u^{q}\mathrm{d}\sigma)$ in the first one,  we obtain
\begin{equation}\label{double inequality wolff}
    u\geq\mathbf{W}_{\alpha,p}\left(\left(\mathbf{W}_{\alpha,p}(u^{q_2}\mathrm{d}\sigma)\right)^{q_1}\mathrm{d}\sigma\right) \quad \mathrm{d}\sigma \mbox{-a.e in } \mathds{R}^n.
   \end{equation}
Thus, if $(u,v)$ solves \eqref{sistemawolffsuper}, then $u$ solves \eqref{double inequality wolff}. In the next Lemma, we will use this perception to obtain a lower bound for supersolutions of \eqref{sistemawolff} in terms of Wolff potentials. 

\begin{lemma}\label{estimativainferior}
   If $(u,v)$ is a nontrivial supersolution to \eqref{sistemawolff}, then there exist a positive constant $c$, which depends  only on $n,\,p,\,\alpha,\,q_1$ and $q_2$, such that the inequalities below holds
   \begin{equation*}
\begin{aligned}
& u(x) \geq  c\left(\mathbf{W}_{\alpha,p}\sigma(x)\right)^{\frac{(p-1)(p-1+q_1)}{(p-1)^2-q_1q_2}}, &\quad \mathrm{d}\sigma \mbox{-a.e in } \mathds{R}^n,\\
& v(x)\geq  c\left(\mathbf{W}_{\alpha,p}\sigma(x)\right)^{\frac{(p-1)(p-1+q_2)}{(p-1)^2-q_1q_2}}, & \quad \mathrm{d}\sigma \mbox{-a.e in } \mathds{R}^n.
\end{aligned}
   \end{equation*}
\end{lemma}
Before proving Lemma~\ref{estimativainferior}, let us recall the following result \cite[Lemma~3.5]{MR3567503}.

\begin{lemmaletter}
    \label{estimativekappa}
    Let $\omega\in M^+(\mathds{R}^n)$. For every $r>0$ and for all $x\in \mathds{R}^n$, it holds
    \begin{equation}
        \mathbf{W}_{\alpha,p}\left((\mathbf{W}_{\alpha,p}\omega)^r\mathrm{d} \omega\right)(x)\geq \kappa^{\frac{r}{p-1}}\left(\mathbf{W}_{\alpha,p}\omega(x)\right)^{\frac{r}{p-1}+1},
    \end{equation}
    where $\kappa$ depends only on $n, p$ and $\alpha$.
\end{lemmaletter}

\begin{proof}[Proof of Lemma~\ref{estimativainferior}]
We will prove the lower estimate for $u$, the first coordinate of $(u,v)$. The lower estimate for $v$ is entirely analogous. First, we prove the following claim.
\begin{Claim}\label{claim1} Let $\sigma\in M^+(\mathds{R}^n)$ and let $w$ be a nontrivial solution to \eqref{double inequality wolff}. 
If there exist $c>0$ and $\delta>0$ such that 
\begin{equation*}
    w(x)\geq c\left(\mathbf{W}_{\alpha,p}\sigma(x)\right)^{\delta}, \quad \forall  x\in\mathds{R}^n,
\end{equation*}
then  
\begin{equation}
    w(x)\geq c^{\frac{q_1q_2}{(p-1)^2}}\kappa^{\frac{q_1(p-1 + 2q_2 \delta)}{(p-1)^2}}\left(\mathbf{W}_{\alpha,p}\sigma(x)\right)^{\frac{q_1}{p-1}\left(\frac{q_2}{p-1}\delta+1\right)+1}, \quad \forall x \in \mathds{R}^n,
\end{equation}
where $\kappa=\kappa(n,p,\alpha)>0$ is the constant in Lemma \ref{estimativekappa}.
\end{Claim}
\noindent Indeed, if $w\geq c\left(\mathbf{W}_{\alpha,p}\sigma\right)^{\delta}$, we can estimate for all $x\in\mathds{R}^n,$
\begin{align*}
   w(x) & = \mathbf{W}_{\alpha,p}\left(\left(\mathbf{W}_{\alpha,p}(u^{q_2}\mathrm{d} \sigma)\right)^{q_1}\mathrm{d} \sigma\right)(x)\\
    & \geq  c^{\frac{q_1q_2}{(p-1)^2}}\mathbf{W}_{\alpha,p}\left[\left(\mathbf{W}_{\alpha,p}(\mathbf{W}_{\alpha,p}\sigma)^{q_2\delta}\mathrm{d}\sigma)\right)^{q_1}\mathrm{d} \sigma\right](x).
\end{align*}
Using Lemma \ref{estimativekappa} twice with $\omega=\sigma$,  $r=q_2 \delta$ and $r=(\delta q_2/(p-1)+1)q_1$ respectively in the previous estimate, we deduce that
\begin{align*}
    w(x) &  \geq  c^{\frac{q_1q_2}{(p-1)^2}}\mathbf{W}_{\alpha,p}\left[\left(\kappa^{\frac{q_2}{p-1}\delta}\left(\mathbf{W}_{\alpha,p}\sigma\right)^{\frac{q_2}{p-1}\delta+1}\right)^{q_1}\mathrm{d} \sigma\right](x)\\
    & = c^{\frac{q_1q_2}{(p-1)^2}}\kappa^{\frac{q_1q_2}{(p-1)^2}\delta}\mathbf{W}_{\alpha,p}\left[\left(\mathbf{W}_{\alpha,p}\sigma\right)^{\left(\frac{q_2}{p-1}\delta+1\right)q_1}\mathrm{d} \sigma\right](x)
    \\
    & \geq c^{\frac{q_1q_2}{(p-1)^2}}\kappa^{\frac{q_1q_2}{(p-1)^2}\delta}\kappa^{\frac{q_1}{p-1}\left(\frac{q_2}{p-1}\delta+1\right)}\left(\mathbf{W}_{\alpha,p}\sigma(x)\right)^{\frac{q_1}{p-1}\left(\frac{q_2}{p-1}\delta+1\right)+1},
\end{align*}
which completes the proof of Claim~\ref{claim1}.

Now, let us fix arbitrarily $x\in\mathds{R}^n$ and $R>|x|$. Let $B=B(0,R)$, $\sigma_B=\chi_B\sigma$ and $\mu \in M^+(\mathds{R}^n)$ defined by $\mu=u^{q_2}\sigma_B$. From \eqref{double inequality wolff}, 
\begin{align}
        u(x)&\geq \mathbf{W}_{\alpha,p}\left(\left(\mathbf{W}_{\alpha,p}(u^{q_2}\mathrm{d} \sigma)\right)^{q_1}\mathrm{d} \sigma\right)(x) \nonumber\\
        &\geq \mathbf{W}_{\alpha,p}\left(\left(\mathbf{W}_{\alpha,p}\mu\right)^{q_1}\mathrm{d} \sigma\right)(x)  \nonumber\\
        &\geq \mathbf{W}_{\alpha,p}\left(\left(\mathbf{W}_{\alpha,p}\mu\right)^{q_1}\mathrm{d} \sigma_B\right)(x). \label{estimativalemmainferior1}
\end{align}
We first obtain a lower bound for $\mathbf{W}_{\alpha,p}\mu(z)$,
\begin{align*}
   \mathbf{W}_{\alpha,p}\mu(z)&=\int_0^{\infty}\left(\frac{\mu (B(z,t))}{t^{n-\alpha p}}\right)^{\frac{1}{p-1}}\frac{\mathrm{d} t}{t} \geq \int_R^{\infty}\left(\frac{\mu (B(z,t))}{t^{n-\alpha p}}\right)^{\frac{1}{p-1}}\frac{\mathrm{d} t}{t} \\
    &= c_0 \int_R^{\infty}\left(\frac{\mu (B(z,2s))}{s^{n-\alpha p}}\right)^{\frac{1}{p-1}}\frac{\mathrm{d} s}{s},
\end{align*}
where $c_0=2^{-(n-\alpha p)/(p-1)}$. Note that if $z\in B$ and $s\geq R$, $B(0,s)\subset B(z,2s)$. Thus for $z\in B$
\begin{align}
     \mathbf{W}_{\alpha,p}\mu(z)& \geq  2^{-\frac{n-\alpha p}{p-1}} \int_R^{\infty}\left(\frac{\mu (B(0,s))}{s^{n-\alpha p}}\right)^{\frac{1}{p-1}}\frac{\mathrm{d} s}{s} = 2^{-\frac{n-\alpha p}{p-1}} \int_R^{\infty}\left(\frac{\int_{B(0,s)}u^{q_2}\mathrm{d} \sigma_B}{s^{n-\alpha p}}\right)^{\frac{1}{p-1}}\frac{\mathrm{d} s}{s} \nonumber\\
     & =2^{-\frac{n-\alpha p}{p-1}}\int_R^{\infty}\left(\frac{\int_{B(0,s)\cap B}u^{q_2}\mathrm{d} \sigma}{s^{n-\alpha p}}\right)^{\frac{1}{p-1}}\frac{\mathrm{d} s}{s} = 2^{-\frac{n-\alpha p}{p-1}}\int_R^{\infty}\left(\frac{\int_{B}u^{q_2}\mathrm{d} \sigma}{s^{n-\alpha p}}\right)^{\frac{1}{p-1}}\frac{\mathrm{d} s}{s} \nonumber\\
     & = 2^{-\frac{n-\alpha p}{p-1}} \left(\int_{B}u^{q_2}\mathrm{d} \sigma\right)^{\frac{1}{p-1}} \int_R^{\infty}\left(\frac{1}{s^{n-\alpha p}}\right)^{\frac{1}{p-1}}\frac{\mathrm{d} s}{s} =: A(R) \label{estimete A(R)}
\end{align}
Let $c_1= A(R)^{q_1/(p-1)}$. By \eqref{estimativalemmainferior1}, it follows from \eqref{estimete A(R)} that 
\begin{equation}\label{first estimate}
    u(x)\geq c_1\mathbf{W}_{\alpha,p}\sigma_B (x).
\end{equation} 
With the aid of Claim~\ref{claim1}, where we consider $\sigma_B$ in place of $\sigma$ and $\delta=1$, we obtain from \eqref{first estimate} that
\begin{equation*}
     u(x)\geq c_2\left(\mathbf{W}_{\alpha,p}\sigma_B(x)\right)^{{\delta_2}},
\end{equation*}
where 
\begin{equation*}
c_2={c_1}^{\frac{q_1q_2}{(p-1)^2}}\kappa^{\frac{q_1(p-1 + 2q_2 )}{(p-1)^2}}, \quad \delta_2=\frac{q_1}{p-1}\left(\frac{q_2}{p-1}+1\right)+1.
\end{equation*}
In fact, setting $\delta_1=1$ and $c_1$ as above, iterating \eqref{estimativalemmainferior1} and \eqref{first estimate} with Claim~\ref{claim1}, we concluded that
\begin{equation}\label{estimateu_j}
    u(x)\geq c_j\left(\mathbf{W}_{\alpha,p}\sigma_B(x)\right)^{\delta_j}, 
\end{equation}
where $\delta_j$ and $c_j$, for $j=2,3,\ldots$, are given by
\begin{align}
    \delta_j&=\frac{q_1}{p-1}\left(\frac{q_2}{p-1}\delta_{j-1}+1\right)+1,\label{seqr_j} \\
    c_j&= {c_{j-1}}^{\frac{q_1q_2}{(p-1)^2}}\kappa^{\frac{q_1(p-1 + 2q_2 \delta_{j-1})}{(p-1)^2}}.\label{seqc_j}
\end{align}
If $\gamma_1=\lim_{j\to \infty} \delta_j$ and $C=\lim_{j\to \infty}c_j$, letting $j\to \infty$ in \eqref{seqr_j} and \eqref{seqc_j} is a straightforward computation to conclude that
\begin{equation}\label{gamma1}
    \gamma_1=\frac{(p-1)(p-1+q_1)}{(p-1)^2-q_1q_2} \quad \mbox{and}\quad C=\kappa^{\frac{q_1(p-1)[(p-1)^2+2(p-1)q_2+q_1q_2]}{[(p-1)^2-q_1q_2]^2}}.
\end{equation}
Hence letting $j\to\infty$ in \eqref{estimateu_j}, we obtain
\begin{equation}\label{lowerestilocal}
    u(x)\geq C\left(\mathbf{W}_{\alpha,p}\sigma_B(x)\right)^{\gamma_1},
\end{equation}
where we recall that $B=B(0,R)$ and  $R>|x|$.
Finally, letting $R\to \infty$ in \eqref{lowerestilocal} yields 
\begin{equation*}
    u(x)\geq C\left(\mathbf{W}_{\alpha,p}\sigma(x)\right)^{\frac{(p-1)(p-1+q_1)}{(p-1)^2-q_1q_2}}, \quad \forall x\in\mathds{R}^n.
\end{equation*}
The same arguments, replacing $q_1$ by $q_2$ in \eqref{double inequality wolff} and $u$ by $v$, yield the lower estimate for $v$, that is,
\begin{equation}\label{estimatev_j}
    v(x)\geq \tilde{c}_j\left(\mathbf{W}_{\alpha,p}\sigma_B(x)\right)^{\tilde{\delta}_j},
\end{equation}
where 
\begin{align}
  \tilde{\delta}_1&=1, \nonumber \\
  \tilde{c}_1&=2^{-\frac{n-\alpha p}{p-1}} \left(\int_{B}v^{q_1}\mathrm{d} \sigma\right)^{\frac{1}{p-1}} \int_R^{\infty}\left(\frac{1}{s^{n-\alpha p}}\right)^{\frac{1}{p-1}}\frac{\mathrm{d} s}{s},   \nonumber\\
  \tilde{\delta}_j&=\frac{q_2}{p-1}\left(\frac{q_1}{p-1}\tilde{\delta}_{j-1}+1\right)+1,\label{seqdelta_j} \\
    \tilde{c}_j&= {\tilde{c}_{j-1}}^{\frac{q_1q_2}{(p-1)^2}}\kappa^{\frac{q_2(p-1 + 2q_1 \tilde{\delta}_{j-1})}{(p-1)^2}}, \;  j=2,3,\ldots\label{seqd_j}
\end{align}
If $\gamma_2=\lim_{j\to \infty} \tilde{\delta}_j$ and $\Tilde{C}=\lim_{j\to \infty}\tilde{c}_j$, letting $j\to \infty$ in \eqref{seqdelta_j} and \eqref{seqd_j}, a straightforward computation yields 
\begin{equation}\label{gamma2}
    \gamma_2=\frac{(p-1)(p-1+q_2)}{(p-1)^2-q_1q_2} \quad \mbox{and}\quad \tilde{C}=\kappa^{\frac{q_2(p-1)[(p-1)^2+2(p-1)q_1+q_1q_2]}{[(p-1)^2-q_1q_2]^2}}.
\end{equation}
Thus, taking the limits $j\to\infty$ and $R\to\infty$ respectively in \eqref{estimatev_j},
\begin{equation*}
    v(x)\geq \tilde{C}\left(\mathbf{W}_{\alpha,p}\sigma(x)\right)^{\frac{(p-1)(p-1+q_2)}{(p-1)^2-q_1q_2}}, \quad\forall x\in\mathds{R}^n.
\end{equation*}
Setting  $c=\min\{C, \tilde{C}\}$, we complete the proof of Lemma~\ref{estimativainferior}.
\end{proof}

We will prove that conditions \eqref{general wolff finite} and \eqref{sigma abscont alfa p capacidade} are sufficient for the existence of a positive solution to the Syst.~\eqref{sistemawolff}. For that, let us state the following result,  \cite[Lemma~2.1]{MR3556326}.

\begin{lemmaletter}
    \label{regularidade wolff}
       Let $\sigma\in {M}^+(\mathds{R}^n)$ satisfying \eqref{sigma abscont alfa p capacidade}. 
       Then, for every $s>0$, 
       \begin{equation}\label{wolff L^sloc}
           \int_E(\mathbf{W}_{\alpha,p}\sigma_E)^s\,\mathrm{d}\sigma\leq c\, \sigma(E) \quad \mbox{for all compact sets } E\subset \mathds{R}^n,
       \end{equation}
       where $c$ is a positive constant which depends on $n,\;p,\;\alpha,\;s$ and $C_{\sigma}$. Moreover, if \eqref{wolff L^sloc} holds for a given $s>0$, then \eqref{sigma abscont alfa p capacidade} holds with $C=C(n,p,\alpha,s,c)$; hence \eqref{wolff L^sloc} holds for every $s>0$.
\end{lemmaletter}

    Lemma \ref{regularidade wolff} implies, in particular, that for all balls $B$,
    \begin{equation}\label{wolff L^sloc usual}
         \int_B(\mathbf{W}_{\alpha,p}\sigma_B)^s\,\mathrm{d}\sigma\leq c\, \sigma(B).
    \end{equation}
The following Lemma is a consequence of \eqref{wolff L^sloc usual}.
\begin{lemma}\label{regularidade solution wolff}
   Let $\sigma\in {M}^+(\mathds{R}^n)$ satisfying  \eqref{general wolff finite} and \eqref{sigma abscont alfa p capacidade}. Then 
   \begin{equation*}
       \mathbf{W}_{\alpha,p}\sigma\in L_{\mathrm{loc}}^s(\mathds{R}^n,\mathrm{d} \sigma), \quad \forall s>0.
   \end{equation*}
\end{lemma}
\begin{proof}
We need to show that $\mathbf{W}_{\alpha,p}\sigma\in L^s(B, \mathrm{d}\sigma)$ for every ball $B=B(x_0,R)$  in $\mathds{R}^n$ and $s>0$. We set $2B=B(x_0,2R)$ and denote by $(2B)^c$ the complement of $2B$ in $\mathds{R}^n$. 
Using the elementary inequality: given $r>0$, it holds 
\begin{equation}\label{inequality elementary}
    |a_1+a_2|^{r}\leq 2^r\,(|a_1|^r+|a_ 2|^r), \quad \forall a_1, a_2\in\mathds{R},
\end{equation}
and writing $\sigma=\sigma_{2B}+\sigma_{(2B)^c} 
$, we see from \eqref{inequality elementary} that
\begin{align*}
\int_B(\mathbf{W}_{\alpha,p}\sigma)^s\,\mathrm{d}\sigma &\leq c_0\int_B\left(\mathbf{W}_{\alpha,p}\sigma_{2B}\right)^s\,\mathrm{d}\sigma+c_0\int_B\left(\mathbf{W}_{\alpha,p}\sigma_{(2B)^{c}}\right)^s\,\mathrm{d}\sigma\\
    & =: c_0\left(I_1 + I_2\right),
\end{align*}
where $c_0=c_0(p,s)>0$. By hypothesis on $\sigma$, we can using \eqref{wolff L^sloc usual} to estimate $I_1$
\begin{equation*}
        I_1 \leq \int_{2B}\left(\mathbf{W}_{\alpha,p}\sigma_{2B}\right)^s\,\mathrm{d}\sigma \leq c_1 \,\sigma(2B)<\infty,
\end{equation*}
where $c_1=c_1(s,p,c)$, with $c$ given in \eqref{wolff L^sloc usual}. To estimate $I_2$, we first notice that $(2B)^c\cap B(y,t)=\emptyset$ for $y\in B$ and $0<t< R$. Hence, for $y\in B$, we have
\begin{align*}
    \mathbf{W}_{\alpha,p}\sigma_{(2B)^c}(y)&=\int_0^{\infty}\left(\frac{\sigma((2B)^c\cap B(y,t))}{t^{n-\alpha p}}\right)^{\frac{1}{p-1}}\frac{\mathrm{d} t}{t}\\
    &= \int_R^{\infty}\left(\frac{\sigma((2B)^c\cap B(y,t))}{t^{n-\alpha p}}\right)^{\frac{1}{p-1}}\frac{\mathrm{d} t}{t}.
\end{align*}
If $t\geq R$ and $y\in B$, $B(y,t)\subset B(0,2t)$, and consequently
\begin{align*}
    \mathbf{W}_{\alpha,p}\sigma_{(2B)^c}(y)&\leq \int_R^{\infty}\left(\frac{\sigma((2B)^c\cap B(0,2t))}{t^{n-\alpha p}}\right)^{\frac{1}{p-1}}\frac{\mathrm{d} t}{t}\\
    & \leq \int_R^{\infty}\left(\frac{\sigma(B(0,2t))}{t^{n-\alpha p}}\right)^{\frac{1}{p-1}}\frac{\mathrm{d} t}{t} \leq 2^{\frac{n-\alpha p}{p-1}}\int_R^{\infty}\left(\frac{\sigma(B(0,t))}{t^{n-\alpha p}}\right)^{\frac{1}{p-1}}\frac{\mathrm{d} t}{t}. 
\end{align*}
From the previous estimate, we obtain,
\begin{align*}
    \sup_{y\in B}\left(W_{\alpha,p}\sigma_{(2B)^c}(y)\right)^s&\leq 2^{\frac{s(n-\alpha p)}{p-1}}\left(\int_R^{\infty}\left(\frac{\sigma(B(0,t))}{t^{n-\alpha p}}\right)^{\frac{1}{p-1}}\frac{\mathrm{d} t}{t}\right)^s\\
     &=: A(R).
\end{align*}
Thus 
\begin{align*}
   I_2= \int_B \left(W_{\alpha,p}\sigma_{(2B)^c}\right)^s\,\mathrm{d}\sigma \leq A(R)\sigma(B)< \infty.
\end{align*}
This proves that both $I_1$ and $I_2$ are finite, and consequently $\mathbf{W}_{\alpha,p}\sigma\in L^s(B,\mathrm{d}\sigma)$ for every ball $B$ in $\mathds{R}^n$.
\end{proof}

\subsection{Proof of Theorem~\ref{existencewollfsystem}} 
We will proceed by the method of sub- and super-solutions. 
We begin recalling $\gamma_1, \; \gamma_2$  given in \eqref{gamma1} and \eqref{gamma2}, that is
\begin{equation*}
     \gamma_1=\frac{(p-1)(p-1+q_1)}{(p-1)^2-q_1q_2} \quad \mbox{and} \quad
       \gamma_2=\frac{(p-1)(p-1+q_2)}{(p-1)^2-q_1q_2},
\end{equation*}
which can be rewritten as follows 
\begin{equation*}
    \gamma_1=\frac{q_1}{p-1}{\gamma_2}+1 \quad \mbox{and}\quad {\gamma_2}=\frac{q_2}{p-1}\gamma_1+1.
\end{equation*}
\begin{Claim}\label{claim2} There exists $ \lambda_1>0 $ sufficiently small such that
    \begin{equation*}
    (\underline{u},\underline{v})= \big(\lambda_1 (\mathbf{W}_{\alpha,p}\sigma)^{\gamma_1}, \lambda_1 (\mathbf{W}_{\alpha,p}\sigma)^{{\gamma_2}}\big)
\end{equation*}
is a subsolution to \eqref{sistemawolff}.
\end{Claim}
\noindent Indeed, using Lemma \ref{estimativekappa},
\begin{align*}
     \mathbf{W}_{\alpha,p}(\underline{v}^{q_1}\mathrm{d}\sigma)&= {\lambda_1}^{\frac{q_1}{p-1}}\mathbf{W}_{\alpha,p}( (\mathbf{W}_{\alpha,p}\sigma)^{q_1{\gamma_2}}\mathrm{d}\sigma)\geq {\lambda_1}^{\frac{q_1}{p-1}}\kappa^{\frac{q_1}{p-1}{\gamma_2}}(\mathbf{W}_{\alpha,p}\sigma)^{\frac{q_1}{p-1}{\gamma_2}+1} \\
    & = {\lambda_1}^{\frac{q_1}{p-1}}\kappa^{\frac{q_1}{p-1}{\gamma_2}}(\mathbf{W}_{\alpha,p}\sigma)^{\gamma_1},\\    
    \mathbf{W}_{\alpha,p}(\underline{u}^{q_2}\mathrm{d}\sigma)&= {\lambda_1}^{\frac{q_2}{p-1}}\mathbf{W}_{\alpha,p}( (\mathbf{W}_{\alpha,p}\sigma)^{q_2{\gamma_1}}\mathrm{d}\sigma)\geq {\lambda_1}^{\frac{q_2}{p-1}}\kappa^{\frac{q_2}{p-1}{\gamma_1}}(\mathbf{W}_{\alpha,p}\sigma)^{\frac{q_2}{p-1}{\gamma_1}+1} \\
    & = {\lambda_1}^{\frac{q_2}{p-1}}\kappa^{\frac{q_2}{p-1}{\gamma_1}}(\mathbf{W}_{\alpha,p}\sigma)^{\gamma_2}.
\end{align*}
Now, choosing ${\lambda_1}=\min\{\kappa^{(q_1\gamma_2)/(p-1-q_1)},\kappa^{(q_2\gamma_2)/(p-1-q_2)}\}$, we deduce
\begin{equation*}
\begin{aligned}
    & \underline{u}\leq \mathbf{W}_{\alpha,p}(\underline{v}^{q_1}\mathrm{d}\sigma),\\
    & \underline{v}\leq \mathbf{W}_{\alpha,p}(\underline{u}^{q_2}\mathrm{d}\sigma),
\end{aligned}
\end{equation*}
which completes the proof of Claim~\ref{claim2}.

\begin{Claim}\label{claim3}
    There exists $\lambda_2>0$ sufficiently large such that
\begin{equation*}
(\overline{u},\overline{v})=\big(\lambda_2\left(\mathbf{W}_{\alpha,p}\sigma + \left(\mathbf{W}_{\alpha,p}\sigma\right)^{\gamma_1} \right), \lambda_2\left(\mathbf{W}_{\alpha,p}\sigma + \left(\mathbf{W}_{\alpha,p}\sigma\right)^{{\gamma_2}} \right)\big)
\end{equation*}
is a supersolution to \eqref{sistemawolff}, such that $\overline{u}\geq \underline{u}$ and $\overline{v}\geq \underline{v}$.
\end{Claim}
 \noindent Indeed, we begin by construction upper bounds to  $\mathbf{W}_{\alpha,p}(\overline{v}^{q_1}\mathrm{d}\sigma)$ and $\mathbf{W}_{\alpha,p}(\overline{u}^{q_2}\mathrm{d}\sigma)$. For the first one, we need to obtain two estimates, which will be done in two steps. First, arguing as in the proof of Lemma~\ref{regularidade solution wolff}, one can check that there exists $c_1=c_1(p,q_1,\sigma)>0$ such that, for every $x\in\mathds{R}^n$,
 \begin{align}\label{inequalityq1}
     \int_0^{\infty}\Big(\frac{\int_{B(x,t)}(\mathbf{W}_{\alpha,p}\sigma)^{q_1}\,\mathrm{d}\sigma}{t^{n-\alpha p}}\Big)^{\frac{1}{p-1}}\frac{\mathrm{d} t}{t}&\leq \int_0^{\infty}\Big(\frac{c_1\sigma(B(x,2t))+c_1\sigma(B(x,t))}{t^{n-\alpha p}}\Big)^{\frac{1}{p-1}}\frac{\mathrm{d} t}{t} \nonumber\\
     & \leq c_2\int_0^{\infty}\Big(\frac{\sigma(B(x,2t))}{t^{n-\alpha p}}\Big)^{\frac{1}{p-1}}\frac{\mathrm{d} t}{t} +c_2 \, \mathbf{W}_{\alpha,p}\sigma(x) \nonumber\\
     & \leq c_2\, 2^{\frac{n-\alpha p}{p-1}}\int_0^{\infty}\Big(\frac{\sigma(B(x,t))}{t^{n-\alpha p}}\Big)^{\frac{1}{p-1}}\frac{\mathrm{d} t}{t} +c_2 \, \mathbf{W}_{\alpha,p}\sigma(x)\nonumber\\
 & \leq c_3 \mathbf{W}_{\alpha,p}\sigma(x),
 \end{align}
where $c_3=c_3(n,p,q_1,\sigma)>0$.  The second step consists of the following estimate
\begin{align*}
    \int_{B(x,t)}(\mathbf{W}_{\alpha,p}\sigma)^{{\gamma_2} q_1}\,\mathrm{d}\sigma&= \int_{B(x,t)}\Big[\int_0^\infty\Big(\frac{\sigma(B(y,r))}{r^{n-\alpha p}}\Big)^{\frac{1}{p-1}}\,\frac{\mathrm{d} r}{r}\Big]^{{\gamma_2} q_1}\,\mathrm{d}\sigma(y)\\
    & \leq c_4 \int_{B(x,t)}\Big[\int_0^t\Big(\frac{\sigma(B(y,r))}{r^{n-\alpha p}}\Big)^{\frac{1}{p-1}}\,\frac{\mathrm{d} r}{r}\Big]^{{\gamma_2} q_1}\,\mathrm{d}\sigma(y)\\
    &\quad + c_4\int_{B(x,t)}\Big[\int_t^\infty\Big(\frac{\sigma(B(y,r))}{r^{n-\alpha p}}\Big)^{\frac{1}{p-1}}\,\frac{\mathrm{d} r}{r}\Big]^{{\gamma_2} q_1}\,\mathrm{d}\sigma(y) =: c_4\left(I_1+I_2\right),
\end{align*}
where $c_4=c_4(p,q_1,q_2)>0$. For $y\in B(x,t)$ and $r\leq t$, we have $B(y,r)\subset B(x,2t)$, whence
\begin{align*}
    I_1&= \int_{B(x,t)}\Big[\int_0^t\Big(\frac{\sigma(B(y,r))}{r^{n-\alpha p}}\Big)^{\frac{1}{p-1}}\,\frac{\mathrm{d} r}{r}\Big]^{{\gamma_2} q_1}\,\mathrm{d}\sigma(y)\\
    &\leq\int_{B(x,2t)}\Big[\int_0^t\Big(\frac{\sigma(B(y,r)\cap B(x,2t))}{r^{n-\alpha p}}\Big)^{\frac{1}{p-1}}\,\frac{\mathrm{d} r}{r}\Big]^{{\gamma_2} q_1}\,\mathrm{d}\sigma(y) \\
    & \leq\int_{B(x,2t)}\Big[\mathbf{W}_{\alpha,p}\sigma_{B(x,2t)}\Big]^{{\gamma_2} q_1}\,\mathrm{d}\sigma(y).
\end{align*}
By \eqref{wolff L^sloc usual}, we deduce
\begin{equation*}
    I_1\leq \int_{B(x,2t)}\left[\mathbf{W}_{\alpha,p}\sigma_{B(x,2t)}\right]^{{\gamma_2} q_1}\,\mathrm{d}\sigma(y)\leq c_5\,\sigma(B(x,2t)),
\end{equation*}
 where $c_5=c_5(n,p,q_1,q_2,C_{\sigma})$. Now, for $r\geq t$, we have $B(y,r)\subset B(x,2r)$, and consequently
\begin{align*}
    I_2&\leq \int_{B(x,t)}\Big[\int_t^\infty\Big(\frac{\sigma(B(x,2r))}{r^{n-\alpha p}}\Big)^{\frac{1}{p-1}}\,\frac{\mathrm{d} r}{r}\Big]^{{\gamma_2} q_1}\,\mathrm{d}\sigma(y) \\
    & = \sigma(B(x,t))\Big[\int_t^\infty\Big(\frac{\sigma(B(x,2r))}{r^{n-\alpha p}}\Big)^{\frac{1}{p-1}}\,\frac{\mathrm{d} r}{r}\Big]^{{\gamma_2} q_1}\\
    &\leq \sigma(B(x,t))\Big[\int_0^\infty\Big(\frac{\sigma(B(x,2r))}{r^{n-\alpha p}}\Big)^{\frac{1}{p-1}}\,\frac{\mathrm{d} r}{r}\Big]^{{\gamma_2} q_1}\leq c_6\,\sigma(B(x,t))\left[\mathbf{W}_{\alpha,p}\sigma (x)\right]^{{\gamma_2} q_1},
\end{align*}
where $c_6=c_6(n,p,\alpha,q_1,q_2)=2^{((n-\alpha p)\gamma_2q_1)/(p-1)}$. From this, we obtain
\begin{equation*}
   \int_{B(x,t)}(\mathbf{W}_{\alpha,p}\sigma)^{{\gamma_2} q_1}\,\mathrm{d}\sigma\leq c_7\left[\sigma(B(x,2t))+\left(\mathbf{W}_{\alpha,p}\sigma (x)\right)^{{\gamma_2} q_1}\sigma(B(x,t))\right], 
\end{equation*}
where $c_7=c_7(n,p,q_1,q_2,\alpha, C_{\sigma})$. Thus, a combination of \eqref{inequality elementary}, \eqref{wolff L^sloc usual}, \eqref{inequalityq1} with the previous inequality yield
\begin{align}
    \mathbf{W}_{\alpha,p}(\overline{v}^{q_1}\mathrm{d}\sigma)(x)&\leq {\lambda_2}^{\frac{q_1}{p-1}}c_8\int_0^{\infty}\Big(\frac{\int_{B(x,t)}(\mathbf{W}_{\alpha,p}\sigma)^{q_1}+(\mathbf{W}_{\alpha,p}\sigma)^{{\gamma_2}q_1}\mathrm{d}\sigma}{t^{n-\alpha p}}\Big)^{\frac{1}{p-1}}\frac{\mathrm{d} t}{t} \nonumber\\
    &\leq {\lambda_2}^{\frac{q_1}{p-1}}c_9\Big[\mathbf{W}_{\alpha,p}\sigma(x)+ \int_0^{\infty}\Big(\frac{\sigma(B(x,2t))}{t^{n-\alpha p}}\Big)^{\frac{1}{p-1}}\frac{\mathrm{d} t}{t}\Big.\nonumber\\
    &\qquad\qquad \Big.+(\mathbf{W}_{\alpha,p}\sigma(x))^{\frac{q_1}{p-1}{\gamma_2}}\int_0^{\infty}\Big(\frac{\sigma(B(x,t))}{t^{n-\alpha p}}\Big)^{\frac{1}{p-1}}\frac{\mathrm{d} t}{t} \Big]\nonumber\\
    &\leq {\lambda_2}^{\frac{q_1}{p-1}}c_{10}\Big[\mathbf{W}_{\alpha,p}\sigma(x)+(\mathbf{W}_{\alpha,p}\sigma(x))^{\frac{q_1}{p-1}{\gamma_2}+1}\Big], \label{upper bound overline v 1.1}
\end{align}
where $c_{10}=c_{10}(n,p,q_1,q_2,\alpha,{\sigma})$. 
Now, to estimate $\mathbf{W}_{\alpha,p}(\overline{u}^{q_2}\mathrm{d}\sigma)$, we use a similar argument. By replacing $q_1$ by $q_2$ and $\gamma_2$ by $\gamma_1$, we deduce
\begin{equation}\label{similarcalculus}
   \begin{aligned}
       &\int_0^{\infty}\Big(\frac{\int_{B(x,t)}(\mathbf{W}_{\alpha,p}\sigma)^{q_2}\,\mathrm{d}\sigma}{t^{n-\alpha p}}\Big)^{\frac{1}{p-1}}\frac{\mathrm{d} t}{t}\leq \tilde{c}_1\mathbf{W}_{\alpha,p}\sigma(x), \\
       & \int_{B(x,t)}(\mathbf{W}_{\alpha,p}\sigma)^{{\gamma_1} q_2}\,\mathrm{d}\sigma \leq \tilde{c}_2\, (\tilde{I}_1+\tilde{I}_2),
   \end{aligned}
\end{equation}
where
\begin{align*}
    \tilde{I}_1 &= \int_{B(x,t)}\Big[\int_0^t\Big(\frac{\sigma(B(y,r))}{r^{n-\alpha p}}\Big)^{\frac{1}{p-1}}\,\frac{\mathrm{d} r}{r}\Big]^{{\gamma_1} q_2}\,\mathrm{d}\sigma(y)\\
    &= \int_{B(x,t)}\Big[\int_0^t\Big(\frac{\sigma(B(y,r)\cap B(x,2t))}{r^{n-\alpha p}}\Big)^{\frac{1}{p-1}}\,\frac{\mathrm{d} r}{r}\Big]^{{\gamma_1} q_2}\,\mathrm{d}\sigma(y)\\
    & \leq\int_{B(x,2t)}\Big[\mathbf{W}_{\alpha,p}\sigma_{B(x,2t)}\Big]^{{\gamma_1} q_2}\,\mathrm{d}\sigma(y)\leq \tilde{c}_3\,\sigma(B(x,2t))
\end{align*}
and 
\begin{align*}
    \tilde{I}_2& = \int_{B(x,t)}\Big[\int_t^\infty\Big(\frac{\sigma(B(y,r))}{r^{n-\alpha p}}\Big)^{\frac{1}{p-1}}\,\frac{\mathrm{d} r}{r}\Big]^{{\gamma_1} q_2}\,\mathrm{d}\sigma(y) \\
    &\leq \int_{B(x,t)}\Big[\int_t^\infty\Big(\frac{\sigma(B(x,2r))}{r^{n-\alpha p}}\Big)^{\frac{1}{p-1}}\,\frac{\mathrm{d} r}{r}\Big]^{{\gamma_1} q_2}\,\mathrm{d}\sigma(y) \\
    & \leq \tilde{c}_4\,\sigma(B(x,t))\left[\mathbf{W}_{\alpha,p}\sigma (x)\right]^{{\gamma_1} q_2},
\end{align*}
which $\tilde{c}_1,\,\tilde{c}_2,\,\tilde{c}_3$ and $\tilde{c}_4$ are constants depending only on $n,\,p,\,q_1,\,q_2$ and $\sigma$. The previous estimates in combination with \eqref{similarcalculus} yield
\begin{equation*}
   \int_{B(x,t)}(\mathbf{W}_{\alpha,p}\sigma)^{{\gamma_1} q_2}\,\mathrm{d}\sigma\leq \tilde{c}_5\left[\sigma(B(x,2t))+\left(\mathbf{W}_{\alpha,p}\sigma (x)\right)^{{\gamma_1} q_2}\sigma(B(x,t))\right], 
\end{equation*}
where $\tilde{c}_5=\tilde{c}_5(n,p,q_1,q_2,\alpha, {\sigma})$.
Hence, using again \eqref{inequality elementary}, \eqref{wolff L^sloc usual} and the previous inequality, we deduce 
\begin{equation}
    \mathbf{W}_{\alpha,p}(\overline{u}^{q_2}\mathrm{d}\sigma)(x)\leq {\lambda_2}^{\frac{q_2}{p-1}}\tilde{c}_6\Big[\mathbf{W}_{\alpha,p}\sigma(x)+(\mathbf{W}_{\alpha,p}\sigma(x))^{\frac{q_1}{p-1}{\gamma_2}+1}\Big], \label{upper bound overline u 1.1}
\end{equation}
where $\tilde{c}_6=\tilde{c}_6(n,p,q_1,q_2,\alpha,{\sigma})$. We recall that $(q_1\gamma_2)/(p-1)+1=\gamma_1$ and $(q_2\gamma_1)/(p-1)+1=\gamma_2$. Therefore, picking $\lambda_2$ such that 
\begin{equation*}
    \lambda_2=\max\{{c_{10}}^{\frac{p-1}{p-1-q_1}},(\tilde{c}_6)^{\frac{p-1}{p-1-q_2}}, \lambda_1\},
\end{equation*}
we conclude from \eqref{upper bound overline v 1.1} and \eqref{upper bound overline u 1.1} that 
\begin{equation*}
\begin{aligned}
    & \overline{u}\geq \mathbf{W}_{\alpha,p}(\overline{v}^{q_1}\mathrm{d}\sigma), &  \overline{u}\geq \underline{u},\\
    &\overline{v}\geq \mathbf{W}_{\alpha,p}(\overline{u}^{q_2}\mathrm{d}\sigma), &\overline{v}\geq \underline{v},
\end{aligned}
\end{equation*}
which completes the proof of Claim~\ref{claim3}.

Now, in order to obtain solutions to \eqref{sistemawolff}, we use a standard iteration argument and Monotone Convergence Theorem. For convenience, we repeat the main idea. Let $u_0=\underline{u}= \lambda_1(\mathbf{W}_{\alpha,p}\sigma)^{\gamma_1}$ and $v_0=\underline{v}=\lambda_1(\mathbf{W}_{\alpha,p}\sigma)^{{\gamma_2}}$, where $\lambda_1$ is the constant obtained in Claim~\ref{claim2}. Clearly, $u_0\leq \overline{u}$ and $v_0\leq \overline{v}$. We set $u_1=\mathbf{W}_{\alpha,p}(v_0^{q_1}\mathrm{d}\sigma)$ and $v_1=\mathbf{W}_{\alpha,p}(u_0^{q_2}\mathrm{d}\sigma)$. By Claim~\ref{claim2}, we have $u_1\geq u_0$ and $v_1\geq v_0$. Let us construct the sequence of pair of functions $(u_j,v_j)$  in $\mathds{R}^n$, with $(u_j,v_j)\in L_{\mathrm{loc}}^{q_2}(\mathds{R}^n,\mathrm{d}\sigma)\times L_{\mathrm{loc}}^{q_1}(\mathds{R}^n,\mathrm{d}\sigma)$ such that
\begin{equation}\label{constructsolutionwollfsistems}
\left\{\begin{aligned}
&u_j=\mathbf{W}_{\alpha,p}(v_{j-1}^{q_1}\mathrm{d} \sigma) \quad \mbox{in}\quad \mathds{R}^n,\\
& v_j=\mathbf{W}_{\alpha,p}(u_{j-1}^{q_1}\mathrm{d} \sigma)\quad \mbox{in}\quad \mathds{R}^n, 
\end{aligned}\right.    
\end{equation}
Indeed, by induction, we can show that the sequences $\{u_j\}$ and $\{v_j\}$ are nondecreasing, with $\underline{u}\leq u_j\leq\overline{u}$ and $\underline{v}\leq v_j\leq\overline{v}$ (for $j=0,1,\ldots$). Due to Lemma~\ref{regularidade solution wolff}, both $(\underline{u}, \underline{v})$ and $(\overline{u},\overline{v})$ belong to $L_{\mathrm{loc}}^{q_2}(\mathds{R}^n,\mathrm{d}\sigma)\times L_{\mathrm{loc}}^{q_1}(\mathds{R}^n,\mathrm{d}\sigma)$, whence $(u_j,v_j)\in L_{\mathrm{loc}}^{q_2}(\mathds{R}^n,\mathrm{d}\sigma)\times L_{\mathrm{loc}}^{q_1}(\mathds{R}^n,\mathrm{d}\sigma)$ for all $j=1, 2, 3, \ldots$. 

Using the  Monotone Convergence Theorem and passing to the limit as $j\to \infty$ in \eqref{constructsolutionwollfsistems}, we see that there exist nonnegative functions $u=\lim u_j$ and $v=\lim v_j$ such that $(u,v)\in L_{\mathrm{loc}}^{q_2}(\mathds{R}^n,\mathrm{d}\sigma)\times L_{\mathrm{loc}}^{q_1}(\mathds{R}^n,\mathrm{d}\sigma)$, which the pair $(u,v)$ satisfies \eqref{sistemawolff} with $\underline{u}\leq u\leq\overline{u}$ and $\underline{v}\leq v\leq\overline{v}$. Thus for a appropriate constant $c=c(n,p,\alpha,q_1,q_2,C_{\sigma})$ we obtain that $u$ and $v$ satisfies \eqref{estimate upper and lower}. This completes the proof of Theorem \ref{existencewollfsystem}.

\subsection{Proof of Theorem~\ref{existencewolffsystemlocweaker}}

As in the proof of the previous Theorem,  
we will proceed by the method of sub- and super-solutions. Let $\underline{u}=\lambda_1 (\mathbf{W}_{\alpha,p}\sigma)^{\gamma_1}$ and $\underline{v}=\lambda_1 (\mathbf{W}_{\alpha,p}\sigma)^{{\gamma_2}}$, where $\gamma_1$ and $\gamma_2$ are given in \eqref{gamma1} and \eqref{gamma2}, respectively. By Claim~\ref{claim2}, we already know that $(\underline{u},\underline{v})$ is a subsolution to \eqref{sistemawolff} if $\lambda_1>0$ is picked to be sufficiently small. It remains now exhibits a supersolution under assumption \eqref{weakercontidion}.

\begin{Claim}\label{claim4}
   There exists $\lambda_2>0$ sufficiently large such that
\begin{equation*}
(\overline{u},\overline{v})=\big(\lambda_2\left(\mathbf{W}_{\alpha,p}\sigma + \left(\mathbf{W}_{\alpha,p}\sigma\right)^{\gamma_1} \right), \lambda_2\left(\mathbf{W}_{\alpha,p}\sigma + \left(\mathbf{W}_{\alpha,p}\sigma\right)^{{\gamma_2}} \right)\big)
\end{equation*}
is a supersolution to \eqref{sistemawolff}, which $\overline{u}\geq \underline{u}$ and $\overline{v}\geq \underline{v}$.
\end{Claim}
\noindent Indeed, using  \eqref{inequality elementary} and assumption \eqref{weakercontidion}, we have the following estimate for $\mathbf{W}_{\alpha,p}(\overline{v}^{q_1}\mathrm{d} \sigma)$:
\begin{align}
  \mathbf{W}_{\alpha,p}{(\overline{v}^{q_1}\mathrm{d}\sigma)}  & \leq c_1 {\lambda_2}^{\frac{q_1}{p-1}} \mathbf{W}_{\alpha,p}\left(\left(\mathbf{W}_{\alpha,p}\sigma\right)^{q_1}\mathrm{d} \sigma\right)+ c_1 {\lambda_2}^{\frac{q_1}{p-1}}\mathbf{W}_{\alpha,p}\left(\left(\mathbf{W}_{\alpha,p}\sigma\right)^{\gamma_2q_1}\mathrm{d} \sigma\right)\nonumber \\
    & \leq c_1 {\lambda_2}^{\frac{q_1}{p-1}}\mathbf{W}_{\alpha,p}\left(\left(\mathbf{W}_{\alpha,p}\sigma\right)^{q_1}\mathrm{d} \sigma\right)+ c_1 {\lambda_2}^{\frac{q_1}{p-1}}\lambda\left(\mathbf{W}_{\alpha,p}\sigma+(\mathbf{W}_{\alpha,p}\sigma)^{\gamma_1} \right) \label{estimateweakerv},
\end{align}
where $c_1=c_1(p,q_1)>0$. In order to establish a convenient upper bound to $\mathbf{W}_{\alpha,p}{(\overline{v}^{q_1}\mathrm{d}\sigma)}$, we need to estimate the first term in the previous inequality.
By H\"{o}lder's inequality and Young's inequality with the exponent $\gamma_2$ and its conjugate
\begin{equation*}
    \gamma_2'=\frac{(p-1)(p-1+q_2)}{q_2(p-1+q_1)},
\end{equation*}
we obtain
\begin{align}
   \int_{B(x,t)}(\mathbf{W}_{\alpha,p}\sigma)^{q_1}\,\mathrm{d} \sigma & \leq \Big(\int_{B(x,t)}(\mathbf{W}_{\alpha,p}\sigma)^{\gamma_2q_1}\,\mathrm{d} \sigma\Big)^{\frac{1}{\gamma_2}} \left[\sigma(B(x,t))\right]^{\frac{1}{\gamma_2'}} \nonumber \\
    & \leq c_2\Big(\int_{B(x,t)}(\mathbf{W}_{\alpha,p}\sigma)^{\gamma_2q_1}\,\mathrm{d} \sigma+\sigma(B(x,t))\Big), \label{estimate1}
\end{align}
where $c_2=c_2(p,q_1,q_2)>0$. 

From \eqref{wolff potential} we can write for $x\in\mathds{R}^n$
\begin{equation*}
    \mathbf{W}_{\alpha,p}\big(\big(\mathbf{W}_{\alpha,p}\sigma\big)^{q_1}\mathrm{d} \sigma\big)(x)=\int_{0}^{\infty}\Big(\frac{\int_{B(x,t)}(\mathbf{W}_{\alpha,p}\sigma)^{q_1}\,\mathrm{d} \sigma}{t^{n-\alpha p}}\Big)^{\frac{1}{p-1}}\frac{\mathrm{d} t}{t},
\end{equation*}
which together with  \eqref{inequality elementary} and \eqref{estimate1} implies 
\begin{multline*}
    \mathbf{W}_{\alpha,p}\big(\big(\mathbf{W}_{\alpha,p}\sigma\big)^{q_1}\mathrm{d} \sigma\big)(x) \leq \\ c_3 \Big[\int_{0}^\infty\Big(\frac{\int_{B(x,t)}(\mathbf{W}_{\alpha,p}\sigma)^{\gamma_2q_1}\,\mathrm{d} \sigma}{t^{n-\alpha p}}\Big)^{\frac{1}{p-1}}\frac{\mathrm{d} t}{t}+\int_{0}^{\infty}\Big(\frac{\sigma(B(x,t))}{t^{n-\alpha p}}\Big)^{\frac{1}{p-1}}\frac{\mathrm{d} t}{t}\Big].
\end{multline*}
Using condition \eqref{weakercontidion} in the previous inequality, where $c_3=c_3(p,q_1,q_2)>0$, we deduce 
\begin{align*}
    \mathbf{W}_{\alpha,p}\big(\big(\mathbf{W}_{\alpha,p}\sigma\big)^{q_1}\mathrm{d} \sigma\big)(x)& \leq c_3 \left(\mathbf{W}_{\alpha,p}\left((\mathbf{W}_{\alpha,p}\sigma)^{\gamma_2q_1}\,\mathrm{d} \sigma\right)(x) +\mathbf{W}_{\alpha,p}\sigma(x)\right) \\
    & \leq c_3\mathbf{W}_{\alpha,p}\sigma(x) +   c_3\lambda \left(\mathbf{W}_{\alpha,p}\sigma(x) + (\mathbf{W}_{\alpha,p}\sigma(x))^{\gamma_1}\right).
\end{align*}
 The last
estimate in combination with  \eqref{estimateweakerv} yield
\begin{multline*}
    \mathbf{W}_{\alpha,p}{(\overline{v}^{q_1}\mathrm{d}\sigma)} \leq \\
    {\lambda_2}^{\frac{q_1}{p-1}}c_1\left[c_3 \mathbf{W}_{\alpha,p}\sigma +   c_3\lambda \left(\mathbf{W}_{\alpha,p}\sigma + (\mathbf{W}_{\alpha,p}\sigma)^{\gamma_1}\right)\right] +c_1 {\lambda_2}^{\frac{q_1}{p-1}}\lambda\left(\mathbf{W}_{\alpha,p}\sigma+(\mathbf{W}_{\alpha,p}\sigma)^{\gamma_1} \right). 
\end{multline*}
Thus, choosing $c_4=c_1(c_3+c_3\lambda+\lambda)$, we concluded that 
\begin{equation}\label{upper bound overline v 1.3}
    \mathbf{W}_{\alpha,p}{(\overline{v}^{q_1}\mathrm{d}\sigma)}\leq {\lambda_2}^{\frac{q_1}{p-1}}c_4\left(\mathbf{W}_{\alpha,p}\sigma+(\mathbf{W}_{\alpha,p}\sigma)^{\gamma_1}\right),
\end{equation}
Now, by a similar argument, we establish an upper bound to $\mathbf{W}_{\alpha,p}(\overline{u}^{q_2}\mathrm{d}\sigma)$. Again, we can replace $q_1$ by $q_2$ and $\gamma_2$ by $\gamma_1$, to obtain from \eqref{inequality elementary} and \eqref{weakercontidion} that

\begin{align}
  \mathbf{W}_{\alpha,p}{(\overline{u}^{q_2}\mathrm{d}\sigma)}  & \leq \tilde{c}_1 {\lambda_2}^{\frac{q_2}{p-1}} \mathbf{W}_{\alpha,p}\left(\left(\mathbf{W}_{\alpha,p}\sigma\right)^{q_2}\mathrm{d} \sigma\right)+ \tilde{c}_1 {\lambda_2}^{\frac{q_2}{p-1}}\mathbf{W}_{\alpha,p}\left(\left(\mathbf{W}_{\alpha,p}\sigma\right)^{\gamma_1q_2}\mathrm{d} \sigma\right)\nonumber \\
    & \leq \tilde{c}_1 {\lambda_2}^{\frac{q_2}{p-1}}\mathbf{W}_{\alpha,p}\left(\left(\mathbf{W}_{\alpha,p}\sigma\right)^{q_2}\mathrm{d} \sigma\right)+ \tilde{c}_1 {\lambda_2}^{\frac{q_2}{p-1}}\lambda\left(\mathbf{W}_{\alpha,p}\sigma+(\mathbf{W}_{\alpha,p}\sigma)^{\gamma_2} \right) \label{estimateweakeru},
\end{align}
where $\tilde{c}_1=\tilde{c}_1(p,q_2)>0$. We also have 
\begin{equation}\label{estimate2}
\mathbf{W}_{\alpha,p}\left(\Big(\mathbf{W}_{\alpha,p}\sigma\right)^{q_2}\mathrm{d} \sigma\Big) \leq \tilde{c}_2 \mathbf{W}_{\alpha,p}\sigma +   \tilde{c}_2\lambda \left(\mathbf{W}_{\alpha,p}\sigma + (\mathbf{W}_{\alpha,p}\sigma)^{\gamma_2}\right),
\end{equation}
for some constant $\tilde{c}_2=\tilde{c}_2(p,q_1,q_2)>0$. From \eqref{weakercontidion}, \eqref{estimateweakeru} and \eqref{estimate2}, it follows
\begin{multline*}
    \mathbf{W}_{\alpha,p}{(\overline{u}^{q_2}\mathrm{d}\sigma)} \leq\\ {\lambda_2}^{\frac{q_2}{p-1}}\tilde{c}_1\left[\tilde{c}_2\mathbf{W}_{\alpha,p}\sigma +   \tilde{c}_2\lambda \left(\mathbf{W}_{\alpha,p}\sigma + (\mathbf{W}_{\alpha,p}\sigma)^{\gamma_2}\right)\right] +\tilde{c}_1 {\lambda_2}^{\frac{q_2}{p-1}}\lambda\left(\mathbf{W}_{\alpha,p}\sigma+(\mathbf{W}_{\alpha,p}\sigma)^{\gamma_2} \right). 
\end{multline*}
Thus, choosing $\tilde{c}_3=\tilde{c}_1(\tilde{c}_2+\tilde{c}_2\lambda+\lambda)$, we concluded that
\begin{equation}\label{upper bound overline u 1.3}
    \mathbf{W}_{\alpha,p}{(\overline{u}^{q_2}\mathrm{d}\sigma)} \leq {\lambda_2}^{\frac{q_2}{p-1}}\tilde{c}_3\left(\mathbf{W}_{\alpha,p}\sigma+(\mathbf{W}_{\alpha,p}\sigma)^{\gamma_2}\right).
\end{equation}
Therefore, picking $\lambda_2$ such that
\begin{equation*}
    \lambda_2=\max\{{c_4}^{\frac{p-1}{p-1-q_1}},(\tilde{c}_3)^{\frac{p-1}{p-1-q_2}},\lambda_1\},
\end{equation*}
we finally see from \eqref{upper bound overline v 1.3} and \eqref{upper bound overline u 1.3} that
\begin{equation*}
    \begin{aligned}
    & \overline{u}\geq \mathbf{W}_{\alpha,p}(\overline{v}^{q_1}\mathrm{d}\sigma), &  \overline{u}\geq \underline{u},\\
    &\overline{v}\geq \mathbf{W}_{\alpha,p}(\overline{u}^{q_2}\mathrm{d}\sigma), &\overline{v}\geq \underline{v},
\end{aligned}
\end{equation*}
which completes the proof of Claim~\ref{claim4}. 
\medskip

Using iterations as in \eqref{constructsolutionwollfsistems}, and the Monotone Convergence Theorem, we ensure that there exists a solution $(u,v)$ to \eqref{sistemawolff} which satisfies \eqref{estimate upper and lower}, with $c=c(n,p,q_1,q_2,\alpha,\lambda)$.
\medskip

Conversely, suppose that there exists a nontrivial solution $(u,v)$ to \eqref{sistemawolff} such that \eqref{estimate upper and lower} holds. By the lower bounds in \eqref{estimate upper and lower}, we have 
\begin{align*}
    & u=\mathbf{W}_{\alpha,p}(v^{q_1}\mathrm{d}\sigma)\geq (c^{-1})^\frac{q_1}{p-1}\mathbf{W}_{\alpha,p}\Big((\mathbf{W}_{\alpha,p}\sigma)^{\gamma_2 q_1}\mathrm{d}\sigma\Big), \\
    &  v=\mathbf{W}_{\alpha,p}(u^{q_2}\mathrm{d}\sigma)\geq (c^{-1})^\frac{q_2}{p-1}\mathbf{W}_{\alpha,p}\Big((\mathbf{W}_{\alpha,p}\sigma)^{\gamma_1 q_2}\mathrm{d}\sigma\Big),
\end{align*}
where $c>0$ is a constant. The previous estimates, in combination with the upper bounds in \eqref{estimate upper and lower}, yield
\begin{align*}
   & \mathbf{W}_{\alpha,p}\left((\mathbf{W}_{\alpha,p}\sigma)^{\gamma_2q_1} \mathrm{d}\sigma\right)\leq \lambda\left(\mathbf{W}_{\alpha,p}\sigma + \left(\mathbf{W}_{\alpha,p}\sigma\right)^{\gamma_1} \right) <\infty \mbox{ - a.e.},\\
   & \mathbf{W}_{\alpha,p}\left((\mathbf{W}_{\alpha,p}\sigma)^{\gamma_1q_2}\mathrm{d}\sigma\right)\leq \lambda\left(\mathbf{W}_{\alpha,p}\sigma + \left(\mathbf{W}_{\alpha,p}\sigma\right)^{\gamma_2} \right) <\infty \mbox{ - a.e.},
\end{align*}
where $\lambda$ can be defined as $\lambda=\max\{{c}^{(p-1+q_1)/(p-1)}, {c}^{(p-1+q_2)/(p-1)}\}$. This completes the proof of Theorem~\ref{existencewolffsystemlocweaker}.

\section{Applications}\label{section4}

In this section, we prove Theorem~\ref{solutionplaplacian general}, Theorem~\ref{solutionplaplacian general weaker}, and Theorem~\ref{thm frac system}. Let us recall some definitions and basic results. First, let us state the local version of Wolff’s inequality in the case $\Omega=\mathds{R}^n$, 
   \begin{equation}\label{wolff's inequality local}
       \mu\in  {M}^+(\mathds{R}^n)\cap W_{\mathrm{loc}}^{-1,p'}(\mathds{R}^n)\Longleftrightarrow \int_B \mathbf{W}_{1,p}\mu_B\,\mathrm{d}\mu < \infty \quad \mbox{for all balls }B,
   \end{equation}
   where $B=B(x,R)$, $\mu_B=\chi_B\mu$ and $W_{\mathrm{loc}}^{-1,p'}(\mathds{R}^n)=W_{\mathrm{loc}}^{1,p}(\mathds{R}^n)^*$ is the dual Sobolev space  (see \cite[Theorem~4.55]{MR1411441} for more details).

In view of Theorem~\ref{constante Maly} and \eqref{wolff's inequality local}, one can check that if $w\in W_{\mathrm{loc}}^{1,p}(\mathds{R}^n)$ solves $-\Delta_p w=\omega$ in the distributional sense with $\omega\in M^+(\mathds{R}^n)$, then $\omega\in W_{\mathrm{loc}}^{-1,p'}(\mathds{R}^n)$. 
We state now a type of converse result proved in \cite[Lemma~3.3]{MR3567503}.
\begin{lemma}\label{soution in W_loc}
Let $1<p<n$ and $\omega\in M^+(\mathds{R}^n) \cap W_{\mathrm{loc}}^{-1,p'}(\mathds{R}^n)$. Suppose that $w$ is a nonnegative $p$-superharmonic solution to 
\begin{equation*}
\left\{\begin{aligned}
       & -\Delta_p w =\omega \quad \mbox{in}\quad \mathds{R}^n,\\
        & \varliminf_{|x|\to \infty}w(x)=0.
    \end{aligned}
     \right.
    \end{equation*}
    Then    $w\in W_{\mathrm{loc}}^{1,p}(\mathds{R}^n)\cap L_{\mathrm{loc}}^{1}(\mathds{R}^n, \mathrm{d}\omega)$.
\end{lemma}

Lemma~\ref{soution in W_loc} in combination with  \eqref{wolff's inequality local}  are be crucial to prove that the solutions to Syst.~\eqref{plaplaciansystem} belong to $W_{\mathrm{loc}}^{1,p}(\mathds{R}^n)\times W_{\mathrm{loc}}^{1,p}(\mathds{R}^n)$.

Now, we recall a basic fact on Wolff's potential (see \cite[Corollary~3.2 (iii)]{MR3567503}): 
If $1<p<\infty$ and $0<\alpha<n/p$, for any $\omega\in M^+(\mathds{R}^n)$, it holds
   \begin{equation}\label{liminf wolff zero}
       \varliminf_{|x|\to\infty} \mathbf{W}_{\alpha, p}\omega(x)=0.
   \end{equation}
Therefore, in view of \eqref{liminf wolff zero}, any nontrivial solution $(u,v)$ to \eqref{sistemawolff}  satisfies
   \begin{equation}\label{liminf wolf solution zero}
       \varliminf_{|x|\to\infty}u(x)=\varliminf_{|x|\to \infty}v(x)=0,
   \end{equation}
provided that $(u,v)$   enjoy the property in \eqref{estimate upper and lower}.
\subsection{Proof of Theorem \ref{solutionplaplacian general}}
First, we assume that $1<p<n$.
The argument is based on the method of successive approximations. 
Suppose \eqref{potencial finito} and \eqref{sigma abs cap_p}. Let $K$ be the constant given in Theorem \ref{constante Maly}.  By Theorem \ref{existencewollfsystem}, with $\alpha=1$ and $K^{p-1}\sigma$ in place of $\sigma$, there exists a nontrivial solution $(\tilde{u},\tilde{v})$ to the system
\begin{equation}\label{auxiliarysystem}
    \left\{
\begin{aligned}
& \tilde{u}= K\, \mathbf{W}_{1,p}(\tilde{v}^{q_1}\mathrm{d} \sigma) \quad \mbox{in}\quad \mathds{R}^n,\\ 
& \tilde{v}= K\, \mathbf{W}_{1,p}(\tilde{u}^{q_2}\mathrm{d} \sigma) \quad \mbox{in}\quad \mathds{R}^n.
\end{aligned}
\right.
\end{equation}
 Since $(\tilde{u},\tilde{v})$ satisfies
\begin{equation}\label{estimative proof p-laplacian}
     \begin{aligned}
    & c^{-1}\Big(K\,\mathbf{W}_{1,p}\sigma\Big)^{\gamma_1}\leq \tilde{u}\leq c\Big(K\,\mathbf{W}_{1,p}\sigma + \Big(K\,\mathbf{W}_{1,p}\sigma\Big)^{\gamma_1} \Big),\\
   & c^{-1}\Big(K\,\mathbf{W}_{1,p}\sigma\Big)^{\gamma_2}\leq \tilde{v}\leq c\Big(K\,\mathbf{W}_{1,p}\sigma + \Big(K\,\mathbf{W}_{1,p}\sigma\Big)^{\gamma_2} \Big),
    \end{aligned}
\end{equation} 
   using \eqref{liminf wolf solution zero}, it follows $\varliminf_{|x|\to\infty}\tilde{u}(x)=\varliminf_{|x|\to\infty}\tilde{v}(x)=0$; here $\gamma_1$ and $\gamma_2$ are as in \eqref{gamma1} and \eqref{gamma2}, respectively. From Lemma~\ref{regularidade solution wolff}, $\tilde{u}$ and  $\tilde{v}$ belong to $L_{\mathrm{loc}}^{s}(\mathds{R}^n,\mathrm{d}\sigma)$ for all $s>0$, hence by H\"{o}lder's inequality 
\begin{align*}
    \int_B\mathbf{W}_{1,p}(\tilde{v}^{q_1}\mathrm{d}\sigma)\tilde{v}^{q_1}\,\mathrm{d}\sigma&= K^{-1}\int_B\tilde{u}\,\tilde{v}^{q_1}\,\mathrm{d}\sigma\\
    & \leq K^{-1}\|\tilde{u}\|_{L^s(B,\,\mathrm{d}\sigma)}\|\tilde{v}^{q_1}\|_{L^{s'}(B,\,\mathrm{d}\sigma)}<\infty,
\end{align*}
for all ball $B\subset\mathds{R}^n$ and, similarly,
\begin{equation*}
    \int_B\mathbf{W}_{1,p}(\tilde{u}^{q_2}\mathrm{d}\sigma)\tilde{u}^{q_2}\,\mathrm{d}\sigma= K^{-1}\|\tilde{u}^{q_2}\|_{L^s(B,\, \mathrm{d}\sigma)}\|\tilde{v}\|_{L^{s'}(B,\, \mathrm{d}\sigma)}<\infty.
\end{equation*}
 Using \eqref{wolff's inequality local}, $\tilde{u}^{q_2}\mathrm{d}\sigma$, $\tilde{v}^{q_1}\mathrm{d}\sigma\in W_{\mathrm{loc}}^{-1,p'}(\mathds{R}^n)$.
By Lemma \ref{estimativainferior}, with $K^{p-1}\sigma$ in place of $\sigma$ again, there exist a constant $c_0=c_0(n,p,\alpha,q_1,q_2)>0$ such that
\begin{equation*}
\begin{aligned}
    &  \tilde{u}\geq c_0\,K^{\gamma_1}(\mathbf{W}_{1,p}\sigma)^{\gamma_1},\\
    & \tilde{v}\geq c_0\,K^{\gamma_2}(\mathbf{W}_{1,p}\sigma)^{\gamma_2}.
\end{aligned}
\end{equation*}

Set $u_0=\varepsilon\,(\mathbf{W}_{1,p}\sigma)^{\gamma_1}$ and $v_0=\varepsilon\,(\mathbf{W}_{1,p}\sigma)^{\gamma_2}$, where $\varepsilon>0$ is a small constant. Using that 
\begin{equation*}
    \gamma_2\frac{q_1}{p-1}+1=\gamma_1 \quad \mbox{and}\quad \gamma_1\frac{q_2}{p-1}+1=\gamma_2.
\end{equation*}
and taking $\varepsilon$ sufficiently small, we obtain $u_0\leq \mathbf{W}_{1,p}(v_0^{q_1}\mathrm{d}\sigma)$ and $v_0\leq\mathbf{W}_{1,p}(u_0^{q_2}\mathrm{d}\sigma)$. Moreover, choosing $\varepsilon < \min\{c\,K^{-\gamma_1},c\,K^{-\gamma_2}\}$, we also have $u_0\leq \tilde{u}$ and $v_0\leq \tilde{v}$, respectively. Setting $B_i=B(0,2^i)$, where $i=1,2,\ldots$, we deduce that $v_0^{q_1}\mathrm{d}\sigma$, $u_0^{q_2}\mathrm{d}\sigma\in W^{-1,p'}(B_i)$, since $v_0\leq\tilde{v}$, $u_0\leq \tilde{u}$ and $\tilde{u}^{q_2}\mathrm{d}\sigma, \tilde{v}^{q_1}\mathrm{d}\sigma\in W_{\mathrm{loc}}^{-1,p'}(\mathds{R}^n)$. Thus, applying \cite[Theorem~21.6]{MR2305115}, there exist unique $p$-superharmonic solutions 
$u_1^i,\, v_1^i\in W_{0}^{1,p}(B_i)$ to equations
\begin{equation*}
    \Delta_p u_1^i=\sigma\, v_0^{q_1}\quad \mbox{in}\quad B_i,\quad \Delta_p v_1^i=\sigma\, u_0^{q_2}\quad \mbox{in}\quad B_i.
\end{equation*}
Using a comparison principle, \cite[Lemma~5.1]{MR3567503}, the sequences $\{u_1^i\}_i$ and $\{v_1^i\}_i$ are increasing. 
We set $u_1=\lim_{i\to \infty}u_1^i$ and $v_1=\lim_{i\to \infty} v_1^i$. A combination of Lemma~\ref{Lemma limit}, of  weak continuity of the $p$-Laplace operator (Theorem~\ref{weak continuity p-laplacian}) and of Monotone Convergence Theorem, ensure that $u_1$ and $v_1$ are $p$-superharmonic solutions to the equations 
\begin{equation*}
 -\Delta_p u_1=\sigma\, v_0^{q_1}\quad \mbox{in}\quad \mathds{R}^n,\quad -\Delta_p v_1=\sigma\, u_0^{q_2}\quad \mbox{in}\quad \mathds{R}^n.
\end{equation*}
Applying Theorem \ref{constante Maly},
\begin{equation*}
    \begin{aligned}
        & u_1^i\leq K\,\mathbf{W}_{1,p}(v_0^{q_1}\mathrm{d}\sigma)\leq K\,\mathbf{W}_{1,p}(\tilde{v}^{q_1}\mathrm{d}\sigma)=\tilde{u},\\
        & v_1^i\leq K\,\mathbf{W}_{1,p}(u_0^{q_2}\mathrm{d}\sigma)\leq K\,\mathbf{W}_{1,p}(\tilde{u}^{q_2}\mathrm{d}\sigma)=\tilde{v},
    \end{aligned}
\end{equation*}
which implies $u_1\leq \tilde{u}$, $v_1\leq \tilde{v}$, and hence by \eqref{liminf wolf solution zero},
\begin{equation*}
    \varliminf_{|x|\to\infty} u_1(x)=\varliminf_{|x|\to \infty}v_1(x)=0.
\end{equation*}
Using the lower bound in Theorem~\ref{constante Maly} and Lemma~\ref{estimativekappa}, we obtain
\begin{align*}
u_1&\geq K^{-1}\mathbf{W}_{1,p}(v_0^{q_1}\mathrm{d}\sigma)=K^{-1}\varepsilon^{\frac{q_1}{p-1}}\mathbf{W}_{1,p}\left((\mathbf{W}_{1,p}\sigma)^{\gamma_2q_1}\mathrm{d}\sigma\right)\\
&\geq K^{-1}\varepsilon^{\frac{q_1}{p-1}} \kappa^{\frac{q_1}{p-1}\gamma_2}(\mathbf{W}_{1,p}\sigma)^{\frac{q_1}{p-1}\gamma_2+1}= K^{-1}\varepsilon^{\frac{q_1}{p-1}} \kappa^{\frac{q_1}{p-1}\gamma_2}(\mathbf{W}_{1,p}\sigma)^{\gamma_1},\\
v_1& \geq K^{-1}\mathbf{W}_{1,p}(u_0^{q_2}\mathrm{d}\sigma)=K^{-1}\varepsilon^{\frac{q_2}{p-1}}\mathbf{W}_{1,p}\left((\mathbf{W}_{1,p}\sigma)^{\gamma_1q_2}\mathrm{d}\sigma\right)\\
&\geq K^{-1}\varepsilon^{\frac{q_2}{p-1}} \kappa^{\frac{q_2}{p-1}\gamma_1}(\mathbf{W}_{1,p}\sigma)^{\frac{q_2}{p-1}\gamma_1+1}= K^{-1}\varepsilon^{\frac{q_2}{p-1}} \kappa^{\frac{q_2}{p-1}\gamma_1}(\mathbf{W}_{1,p}\sigma)^{\gamma_2},
\end{align*}
where $K$ is the constant given in Theorem~\ref{constante Maly} and  $\kappa$ is the constant given in Lemma~\ref{estimativekappa}. Hence, $c_1(\mathbf{W}_{1,p}\sigma)^{\gamma_1}\leq u_1\leq \tilde{u}$ and $\tilde{c}_1(\mathbf{W}_{1,p}\sigma)^{\gamma_2}\leq v_1\leq \tilde{v}$, where
\begin{equation*}
    c_1= K^{-1}\varepsilon^{\frac{q_1}{p-1}} \kappa^{\frac{q_1}{p-1}\gamma_2},\quad \tilde{c}_1=K^{-1}\varepsilon^{\frac{q_2}{p-1}} \kappa^{\frac{q_2}{p-1}\gamma_1}.
\end{equation*}
We notice that $v_0^{q_1}\mathrm{d}\sigma$,  $u_0^{q_2}\mathrm{d}\sigma \in W_{\mathrm{loc}}^{-1,p'}(\mathds{R}^n)$, since  $v_0\leq\tilde{v}$, $u_0\leq \tilde{u}$ and $\tilde{v}^{q_1}\mathrm{d}\sigma$, $\tilde{u}^{q_2}\mathrm{d}\sigma\in W_{\mathrm{loc}}^{-1,p'}(\mathds{R}^n)$. Thus, applying Lemma~\ref{soution in W_loc}, $u_1$, $v_1\in W_{\mathrm{loc}}^{1,p}(\mathds{R}^n)$.

By induction argument, as above, we can construct a sequence $\{(u_j,v_j)\}$ of $p$-superharmonic functions in $\mathds{R}^n$ with $u_j\in L_{\mathrm{loc}}^{q_2}(\mathds{R}^n,\mathrm{d}\sigma)$ and $v_j\in L_{\mathrm{loc}}^{q_1}(\mathds{R}^n,\mathrm{d}\sigma)$ for $j=2,3,\ldots,$ satisfying
\begin{equation}\label{u_j and v_j p-laplacian}
     \left\{
\begin{aligned}
& -\Delta_p u_j=\sigma\, v_{j-1}^{q_1}\quad \mbox{in}\quad \mathds{R}^n,\\ 
&  -\Delta_p v_j=\sigma\, u_{j-1}^{q_2}\quad \mbox{in}\quad \mathds{R}^n,\\
& c_j(\mathbf{W}_{1,p}\sigma)^{\gamma_1}\leq u_j\leq \tilde{u},\ \ \tilde{c}_j(\mathbf{W}_{1,p}\sigma)^{\gamma_2}\leq v_j\leq \tilde{v} \quad \mbox{in}\quad \mathds{R}^n,\\
& 0\leq u_{j-1}\leq u_j, \ \ 0\leq v_{j-1}\leq v_j, \quad u_j, v_j\in W_{\mathrm{loc}}^{1,p}(\mathds{R}^n), \\
& \varliminf_{|x|\to\infty} u_j(x)=\varliminf_{|x|\to \infty}v_j(x)=0,
\end{aligned}
\right.
\end{equation}
where
\begin{equation}\label{c_j and d_j}
    \begin{aligned}
    c_1 &= K^{-1}\varepsilon^{\frac{q_1}{p-1}} \kappa^{\frac{q_1}{p-1}\gamma_2}, &  \tilde{c}_1 &=K^{-1}\varepsilon^{\frac{q_2}{p-1}} \kappa^{\frac{q_2}{p-1}\gamma_1}, &\\
        c_j&=K^{-1}(\tilde{c}_{j-1}\kappa^{\gamma_2})^{\frac{q_1}{p-1}},  &\tilde{c}_j& =K^{-1}(c_{j-1}\kappa^{\gamma_1})^{\frac{q_2}{p-1}},  &\mbox{for }j=2, 3,\ldots.
    \end{aligned}
\end{equation}

\noindent Indeed, suppose that $(u_1,v_1),\ldots,(u_{j-1},v_{j-1})$ have been constructed. Since  $u_{j-1}\leq \tilde{u}$, $v_{j-1}\leq \tilde{v}$ and $\tilde{v}^{q_1}\mathrm{d}\sigma,\,\tilde{u}^{q_2}\mathrm{d}\sigma \in  {M}^{+}(\mathds{R}^n)\cap W_{\mathrm{loc}}^{-1,p'}(\mathds{R}^n)$, it follows $v_{j-1}^{q_1}\mathrm{d}\sigma,\,u_{j-1}^{q_2}\mathrm{d}\sigma \in  {M}^{+}(\mathds{R}^n)\cap W_{\mathrm{loc}}^{-1,p'}(\mathds{R}^n)$. Clearly, $v_{j-1}^{q_1}\mathrm{d}\sigma,\,u_{j-1}^{q_2}\mathrm{d}\sigma \in  W^{-1,p'}(B_i)$, consequently there exist unique $p$-superharmonic solutions $u_j^{i}$ and $v_j^i$ to the equations
\begin{equation*}
     \left\{
\begin{aligned}
& -\Delta_p u_j^i=\sigma\, v_{j-1}^{q_1}\quad \mbox{in}\quad B_i,\qquad u_j^i\in W_{0}^{1,p}(B_i),\\ 
&  -\Delta_p v_j^i=\sigma\, u_{j-1}^{q_2}\quad \mbox{in}\quad B_i,\qquad v_j^i\in W_{0}^{1,p}(B_i).
\end{aligned}
\right.
\end{equation*}
Arguing by induction, let $u_{j-1}^i$ and $v_{j-1}^i$ be the unique solutions of the equations  
\begin{equation*}
     \left\{
\begin{aligned}
& -\Delta_p u_{j-1}^i=\sigma\, v_{j-2}^{q_1}\quad \mbox{in}\quad B_i,\qquad u_{j-1}^i\in W_{0}^{1,p}(B_i),\\ 
&  -\Delta_p v_{j-1}^i=\sigma\, u_{j-2}^{q_2}\quad \mbox{in}\quad B_i,\qquad v_{j-1}^i\in W_{0}^{1,p}(B_i).
\end{aligned}
\right.
\end{equation*}
Since $v_{j-2}\leq v_{j-1}$ and $u_{j-2}\leq u_{j-1}$, it follows from the comparison principle, \cite[Lemma~5.1]{MR3567503}, that $u_{j-1}^i\leq u_{j}^i$ and $v_{j-1}^i\leq v_{j}^i$ for all $i\geq 1$. By Theorem~\ref{constante Maly}, we have
\begin{align*}
    & 0\leq u_{j}^i\leq K\,\mathbf{W}_{1,p}(v_{j-1}^{q_1}\mathrm{d}\sigma)\leq K\,\mathbf{W}_{1,p}(\tilde{v}^{q_1}\mathrm{d}\sigma)=\tilde{u} \quad \mbox{and}\\
    &  0\leq v_{j}^i\leq K\,\mathbf{W}_{1,p}(u_{j-1}^{q_1}\mathrm{d}\sigma)\leq K\,\mathbf{W}_{1,p}(\tilde{u}^{q_1}\mathrm{d}\sigma)=\tilde{v},
\end{align*}
since $v_{j-1}\leq \tilde{v}$ and $u_{j-1}\leq \tilde{u}$. Using again the comparison principle, we deduce that the sequences $\{u_j^i\}_i$ and $\{v_j^i\}_i$ are increasing. 
Thus, letting $u_j=\lim_{i\to \infty} u_j^i$ and $v_j=\lim_{i\to \infty} v_j^i$, we obtain that $u_j$ and $v_j$ are $p$-superharmonic solutions to the equations
\begin{equation*}
     \left\{
\begin{aligned}
& -\Delta_p u_j=\sigma\, v_{j-1}^{q_1}\quad \mbox{in}\quad \mathds{R}^n,\\ 
&  -\Delta_p v_j=\sigma\, u_{j-1}^{q_2}\quad \mbox{in}\quad\mathds{R}^n,
\end{aligned}
\right.
\end{equation*}
since, as before, we apply Theorem~\ref{weak continuity p-laplacian} and the Monotone Convergence Theorem. We also have $u_{j-1}\leq u_j$ and $v_{j-1}\leq v_j$, since $u_{j-1}^i\leq u_j^i$ and $v_{j-1}^i\leq v_j^i$ for all $i\geq 1$. Furthermore, $u_j\leq \tilde{u}$ and $v_j\leq \tilde{v}$, since $u_{j}^i\leq \tilde{u}$ and $v_{j}^i\leq \tilde{v}$ for all $i\geq 1$. By \eqref{liminf wolf solution zero}, we obtain
\begin{equation}\label{liminfujvj}
    \varliminf_{|x|\to\infty} u_j(x)=\varliminf_{|x|\to \infty}v_j(x)=0.
\end{equation}
Since $v_{j-1}^{q_1}\mathrm{d}\sigma$ and $u_{j-1}^{q_2}\mathrm{d}\sigma\in W_{\mathrm{loc}}^{-1,p'}(\mathds{R}^n)$, it follows from Lemma~\ref{soution in W_loc} and \eqref{liminfujvj} that $u_j$, $v_j\in W_{\mathrm{loc}}^{1,p}(\mathds{R}^n)$. Also, applying Theorem~\ref{constante Maly} and Lemma~\ref{estimativekappa}, and arguing by induction, we concluded that
\begin{align*}
    u_j&\geq K^{-1}\mathbf{W}_{1,p}(v_{j-1}^{q_1}\mathrm{d}\sigma)=K^{-1}\mathbf{W}_{1,p}\left({\tilde{c}_{j-1}}^{q_1}(\mathbf{W}_{1,p}\sigma)^{\gamma_2q_1}\mathrm{d}\sigma\right)\\
&\geq K^{-1}{\tilde{c}_{j-1}}^{\frac{q_1}{p-1}} \kappa^{\frac{q_1}{p-1}\gamma_2}(\mathbf{W}_{1,p}\sigma)^{\frac{q_1}{p-1}\gamma_2+1}= c_j(\mathbf{W}_{1,p}\sigma)^{\gamma_1}
\end{align*}
and, similarly,
\begin{align*}
    v_j&\geq K^{-1}\mathbf{W}_{1,p}(u_{j-1}^{q_2}\mathrm{d}\sigma)=K^{-1}\mathbf{W}_{1,p}\left(c_{j-1}^{q_2}(\mathbf{W}_{1,p}\sigma)^{\gamma_1q_2}\mathrm{d}\sigma\right)\\
&\geq K^{-1}c_{j-1}^{\frac{q_2}{p-1}} \kappa^{\frac{q_2}{p-1}\gamma_1}(\mathbf{W}_{1,p}\sigma)^{\frac{q_2}{p-1}\gamma_1+1}= \tilde{c}_j(\mathbf{W}_{1,p}\sigma)^{\gamma_2}.
\end{align*}

Now, we set $u=\lim_{j\to \infty}u_j$ and $v=\lim_{j\to \infty}v_j$. By Lemma~\ref{Lemma limit}, $u$ and $v$ are $p$-superharmonic functions in $\mathds{R}^n$. From Theorem~\ref{weak continuity p-laplacian} and Monotone Convergence Theorem, we deduce that $(u,v)$ is a solution to the system
\begin{equation*}
     \left\{
\begin{aligned}
& -\Delta_p u=\sigma\, v^{q_1}\quad \mbox{in}\quad \mathds{R}^n,\\ 
&  -\Delta_p v=\sigma\, u^{q_2}\quad \mbox{in}\quad\mathds{R}^n,
\end{aligned}
\right.
\end{equation*}
in the sense of \eqref{solutionsystemsense}. Moreover, $u\leq \tilde{u}$ and $v\leq \tilde{v}$, and hence $\varliminf_{|x|\to\infty} u(x)=\varliminf_{|x|\to \infty}v(x)=0$. 
Using Lemma~\ref{soution in W_loc} again, we deduce that $u$, $v\in W_{\mathrm{loc}}^{1,p}(\mathds{R}^n)$ since $v^{q_1}\mathrm{d} \sigma$, $u^{q_2}\mathrm{d}\sigma\in W_{\mathrm{loc}}^{-1,p'}(\mathds{R}^n)$. By \eqref{estimative proof p-laplacian}, there exists a constant $c=c(n,p,q_1,q_2,\alpha,C_{\sigma})>0$ such that
\begin{equation*}
    u\leq c\left(\mathbf{W}_{1,p}\sigma+(\mathbf{W}_{1,p}\sigma)^{\gamma_1}\right), \ v\leq c\left(\mathbf{W}_{1,p}\sigma+(\mathbf{W}_{1,p}\sigma)^{\gamma_2}\right),
\end{equation*}
and this shows an upper bound in \eqref{estimative upper and lower p-laplacian}. From \eqref{u_j and v_j p-laplacian}, $(u,v)$ satisfies for all $j=1,2,\ldots$ the lower bound
    \begin{equation}\label{lowerbound}
    \begin{aligned}
        &  u\geq c_j\left(\mathbf{W}_{1,p}\sigma\right)^{\gamma_1},\\
        & v\geq \tilde{c}_j\left(\mathbf{W}_{1,p}\sigma\right)^{\gamma_2},
    \end{aligned}
    \end{equation}
    Passing to limit $j\to\infty$ in \eqref{lowerbound}, with $c_j$ and $\tilde{c}_j$ given in \eqref{c_j and d_j}, we obtain
   \begin{equation*}
       \begin{aligned}
        &  u\geq C\left(\mathbf{W}_{1,p}\sigma\right)^{\gamma_1},\\
        & v\geq \tilde{C}\left(\mathbf{W}_{1,p}\sigma\right)^{\gamma_2},
    \end{aligned}
   \end{equation*}
   where, by direct computation,
    \begin{equation*}
    \begin{aligned}
        &  C=\lim_{j\to\infty}c_j=K^{-\gamma_1}\kappa^{\frac{\gamma_1q_1(\gamma_1q_2+\gamma_2(p-1))}{(p-1)^2}}, \\
        & \tilde{C}=\lim_{j\to\infty} \tilde{c}_j=K^{-\gamma_2}\kappa^{\frac{\gamma_2q_2(\gamma_2q_1+\gamma_1(p-1))}{(p-1)^2}}.
    \end{aligned}
    \end{equation*}
 This shows a lower bound in \eqref{estimate upper and lower}.

Now we prove that $(u,v)$ is a minimal solution; that is, if $(f,g)$ is any nontrivial solution to Syst.~\eqref{plaplaciansystem}, then $f\geq u$ and $g\geq v$ a.e. Let $(f,g)$ be a solution to Syst.~\eqref{plaplaciansystem}. From Theorem \ref{constante Maly}, it follows
\begin{equation*}
   \left\{ \begin{aligned}
      & f\geq K^{-1}\mathbf{W}_{1,p}(g^{q_1}\mathrm{d}\sigma)  \\
      &g\geq K^{-1}\mathbf{W}_{1,p}(f^{q_2}\mathrm{d}\sigma)
    \end{aligned}
    \right.
\end{equation*} 
Using Lemma \ref{estimativainferior} with $K^{-(p-1)}\sigma$ in place of $\sigma$, we obtain
\begin{equation*}
\begin{aligned}
    & f\geq c\,K^{-\gamma_1}(\mathbf{W}_{1,p}\sigma)^{\gamma_1}, \\
    & g\geq c\,K^{-\gamma_2}(\mathbf{W}_{1,p}\sigma)^{\gamma_2}.
\end{aligned} 
\end{equation*}
Let $\mathrm{d}\omega_1=g^{q_1}\mathrm{d}\sigma$ and $\mathrm{d}\omega_2=f^{q_2}\mathrm{d}\sigma$. Notice that by choice of $\varepsilon$ as above, we have $\mathrm{d}\omega_1\geq v_0^{q_1}\mathrm{d}\sigma$ and $\mathrm{d}\omega_2\geq u_0^{q_2}\mathrm{d}\sigma$. Applying Lemma 5.2 in \cite{MR3567503}, we deduce that the functions $u_1^i$ and $v_1^i$ defined as before satisfy in $B_i$ the inequality $u_1^i\leq f$ and $v_1^i\leq g$ for all $i\geq 1$, and consequently $u_1=\lim_{i\to \infty}u_1^i\leq f$ and $v_1=\lim_{i\to \infty}v_1^i\leq g$. Repeating this argument by induction, we conclude $u_j\leq f$ and $v_j\leq g$ for every $j\geq 1$. Therefore, 
\begin{equation*}
    u=\lim_{j\to \infty}u_j\leq f, \quad v=\lim_{j\to \infty}v_j\leq g \quad  \mbox{a.e. in }\mathds{R}^n.
\end{equation*}


It remains to prove that if $p\geq n$, there are no nontrivial solutions to Syst.~\eqref{plaplaciansystem} in $\mathds{R}^n$.
Let $w\in\mathcal{S}_p(\mathds{R}^n)$ be a $p$-superharmonic function in $\mathds{R}^n$ with $\varliminf_{|x|\to \infty}w(x)=0$.
By \cite[Theorem~7.48]{MR2305115}, there exists $c=c(n,p)>0$ such that
\begin{equation}\label{n less p}
    \int_{B_r}\frac{|\nabla w|^p}{w^p}\,\mathrm{d}x\leq c \,\mathrm{cap}_p(B_r), 
\end{equation}
 for all balls $B_r:=B(0,r)$. With aid of \cite[Theorem~2.2 (ii)]{MR2305115}, we infer from \cite[Theorem~2.19]{MR2305115} that $\mathrm{cap}_p(B_r)=0$ for all balls $B_r$, provided $p\geq n$. Thus, letting $r\to\infty$ in \eqref{n less p}, we obtain for $p\geq n$ that $\nabla w \equiv 0$ almost everywhere on $\mathds{R}^n$. Since $\varliminf_{|x|\to \infty}w(x)=0$, it follows $w\equiv 0$. 
This completes the proof of Theorem~\ref{solutionplaplacian general}.


\subsection{Proof of Theorem \ref{solutionplaplacian general weaker}}
Suppose that there exist $p$-superharmonic functions $u$ and $v$ satisfying Syst.~\eqref{plaplaciansystem} and \eqref{estimative upper and lower p-laplacian}. By Theorem \ref{constante Maly}, 
\begin{equation*}
\left\{\begin{aligned}
    & u\geq K^{-1}\mathbf{W}_{1,p}(v^{q_1}\mathrm{d}\sigma), \\
    & v\geq K^{-1}\mathbf{W}_{1,p}(u^{q_2}\mathrm{d}\sigma),
\end{aligned}
\right.
\end{equation*}
which together with Lemma \ref{estimativainferior} implies 
\begin{equation}\label{lower bound weaker}
    \begin{aligned}
        &  u\geq c_0 \left(\mathbf{W}_{1,p}\sigma\right)^{\gamma_1}, \\
        & v\geq c_0 \left(\mathbf{W}_{1,p}\sigma\right)^{\gamma_2},
    \end{aligned}
\end{equation}
where $\gamma_1$ and $\gamma_2$ are given by \eqref{gamma1} and \eqref{gamma2} respectively.
Consequently,
\begin{equation*}
\begin{aligned}
    & u\geq c_1 \,\mathbf{W}_{1,p}\left((\mathbf{W}_{1,p}\sigma)^{\gamma_2q_1}\mathrm{d}\sigma\right),\\
    & v\geq c_2 \,\mathbf{W}_{1,p}\left((\mathbf{W}_{1,p}\sigma)^{\gamma_1q_2}\mathrm{d}\sigma\right),
\end{aligned}
\end{equation*}
where $c_1=K^{-1}{c_0}^{q_1}$ and $c_2=K^{-1}{c_0}^{q_2}$. 
Thus, in view of \eqref{estimative upper and lower p-laplacian}, it holds 
\begin{equation*}
\begin{aligned}
    & u\leq c_3\left(\mathbf{W}_{1,p}\sigma+(\mathbf{W}_{1,p}\sigma)^{\gamma_1}\right), \\
    & v\leq c_3\left(\mathbf{W}_{1,p}\sigma+(\mathbf{W}_{1,p}\sigma)^{\gamma_2}\right),
\end{aligned}
\end{equation*}
for some constant $c_3>0$. 
The previous estimates in combination with \eqref{lower bound weaker} and choosing $\lambda=c_3\max\{{c_1}^{-1}, {c_2}^{-1}\}$ yield
\begin{align*}
   & \mathbf{W}_{1,p}\left((\mathbf{W}_{1,p}\sigma)^{\gamma_2q_1} \mathrm{d}\sigma\right)\leq \lambda\left(\mathbf{W}_{1,p}\sigma + \left(\mathbf{W}_{1,p}\sigma\right)^{\gamma_1} \right) <\infty \mbox{ - a.e.},\\
   & \mathbf{W}_{1,p}\left((\mathbf{W}_{1,p}\sigma)^{\gamma_1q_2}\mathrm{d}\sigma\right)\leq \lambda\left(\mathbf{W}_{1,p}\sigma + \left(\mathbf{W}_{1,p}\sigma\right)^{\gamma_2} \right) <\infty \mbox{ - a.e.},
\end{align*}
which proves \eqref{condition necessary weaker}.

Conversely, suppose that \eqref{condition necessary weaker} holds. Let $K$ be the constant given in Theorem~\ref{constante Maly}. By Theorem \ref{existencewolffsystemlocweaker} with $\alpha=1$ and $K^{p-1}\sigma$ in place of $\sigma$, there exists a nontrivial solution $(\tilde{u},\tilde{v})$ to Syst.~\eqref{auxiliarysystem}, satisfying \eqref{estimative proof p-laplacian}. 
Using the same argument in the proof of Theorem~\ref{solutionplaplacian general}, one can complete that reciprocal of Theorem~\ref{solutionplaplacian general weaker} holds.

\subsection{Proof of Theorem \ref{thm frac system}}
As was commented previously in Sect.~\ref{intro}, since $\mathbf{I}_{2\alpha}\mu = (n-2\alpha)\mathbf{W}_{\alpha,2}\mu$ for any $\mu\in M^+(\mathds{R}^n)$, Theorem \ref{thm frac system} is a special case of Theorem \ref{existencewollfsystem} with $p=2$.

\section{Final comments}\label{final comments}

\begin{remark}
    For the case $q_1=q_2$, setting $q=q_1$, we obtain
 \begin{equation}\label{case q_1=q_2}
     \gamma_1=\gamma_2=\frac{(p-1)(p-1+q)}{(p-1)^2-q^2}=\frac{p-1}{p-1-q}.
 \end{equation}
 The argument of the proof of Theorem~\ref{solutionplaplacian general} in combination with \eqref{case q_1=q_2} implies $u_0=v_0$. Thus, by induction,
 \begin{equation*}
     u_j=v_j \quad \forall j=1,2,\ldots,
 \end{equation*}
where $(u_j,v_j)$ were given in \eqref{u_j and v_j p-laplacian} for $j=1,2,\ldots$. It follows that $u=v$ and Syst.~\eqref{plaplaciansystem} become the equation
 \begin{equation*}
      \left\{
\begin{aligned}
&  -\Delta_p u=\sigma\,u^{q} \quad \mbox{in}\quad \mathds{R}^n,\\ 
&  \varliminf_{|x|\to \infty}u(x)=0.
\end{aligned}
\right.
   \end{equation*}
Therefore, \cite[Theorem~1.2]{MR3556326} is a corollary of Theorem~\ref{solutionplaplacian general} 
for the case $q=q_1=q_2$.
\end{remark}

\begin{remark}\label{important remark}
We highlight that our main results are entirely based on the Wolff potential estimates. Thus, direct analogous theorems hold for the more general quasilinear $\mathcal{A}$-operator $\mathrm{div}\mathcal{A}(x,\nabla\cdot)$ in place of $\Delta_p\cdot$ in \eqref{kilpelainenmalequation}. Precisely, let $\mathcal{A}:\mathds{R}^n\times \mathds{R}^n\to \mathds{R}^n$ be a mapping satisfying the standard structural assumptions, which ensures the growth condition $\mathcal{A}(x,\xi)\cdot\xi\approx |\xi|^p$. This assumption guarantee
that the Wolff potential estimates like \eqref{potential estimate} hold for the system
   \begin{equation*}
     \left\{
\begin{aligned}
& -\mathrm{div}\mathcal{A}(x,\nabla u)=\sigma\, v^{q_1},\quad v> 0\quad \mbox{in}\quad \mathds{R}^n,\\ 
&   -\mathrm{div}\mathcal{A}(x,\nabla v)=\sigma\, u^{q_2},\quad u>0\quad \mbox{in}\quad\mathds{R}^n,\\
&\varliminf_{|x|\to\infty}u(x)=\varliminf_{|x|\to\infty}v(x)=0,
\end{aligned}
\right.
\end{equation*}
In such wise, we are able to conclude similar theorems  (see \cite{MR1264000, MR3311903, MR1890997} for more details).
\end{remark}

Here we indicate some questions related to this class of quasilinear problems:
\begin{enumerate}
    \item   Let $\mu_1,\,\mu_2\in M^+(\mathds{R}^n)$. Consider the counterpart \textit{inhomogeneous} to \eqref{plaplaciansystem} as follows:

\begin{equation}\label{plaplaciansysteminhomogeneous}
     \left\{
\begin{aligned}
& -\Delta_p u= \sigma\, v^{q_1}+\mu_1, & u>0 \quad \mbox{in}\quad \mathds{R}^n,&\\ 
& -\Delta_p v= \sigma\, u^{q_2}+\mu_2, & v>0 \quad \mbox{in}\quad \mathds{R}^n,&\\
& \varliminf_{|x|\to \infty} u(x)=0, & \varliminf_{|x|\to \infty} v(x)=0. &
\end{aligned}
\right.
 \end{equation}
 Does \eqref{plaplaciansysteminhomogeneous} have a minimal positive solution for any nontrivial measures $\mu_1, \mu_2$? Perhaps, new phenomena involving possible interactions between $\mu_1$, $\mu_2$, and $\sigma$ will occur. 
 
    \item Suppose $\sigma\in M^+(\mathds{R}^n)$ radially symmetric. Thus, the solution to \eqref{fractlaplaciamsystem} obtained in Theorem~\ref{thm frac system} must be radially symmetric. In addition, since $\mathbf{I}_{2\alpha}\sigma=(n-2\alpha)\mathbf{W}_{\alpha, 2}\sigma$, condition \eqref{weakercontidion} can be rewritten in terms of the Riesz potential. Is it possible to characterize condition \eqref{weakercontidion}, in terms of Riesz potential, for $\sigma$ radially symmetric, as was done to the single equation in \cite[Prop.~5.2]{MR3556326}? 
    
    \item  Let us consider the quasilinear elliptic system

\begin{equation}\label{plaplaciansystem differ}
     \left\{
\begin{aligned}
& -\Delta_{p_1} u= \sigma\, v^{q_1}, & u>0 \quad \mbox{in}\quad \mathds{R}^n,&\\ 
& -\Delta_{p_2} v= \sigma\, u^{q_2}, & v>0 \quad \mbox{in}\quad \mathds{R}^n,&\\
& \varliminf_{|x|\to \infty} u(x)=0, & \varliminf_{|x|\to \infty} v(x)=0, &
\end{aligned}
\right.
 \end{equation}  
 where $p_1,\,p_2 \in (1, n]$ and $q_1, \, q_2<\max\{p_1-1, p_2-1\}$.
    Does \eqref{plaplaciansystem differ} have a minimal positive solution? In fact, the question is related to the existence of a nontrivial solution to the following system integral 
    \begin{equation*}\label{sistemawolff differ}
    \left\{
\begin{aligned}
& u= \mathbf{W}_{\alpha_1,p_1}(v^{q_1}\mathrm{d} \sigma), \quad \mathrm{d}\sigma \mbox{-a.e in } \mathds{R}^n,\\
& v= \mathbf{W}_{\alpha_2,p_2}(u^{q_2}\mathrm{d} \sigma), \quad \mathrm{d}\sigma \mbox{-a.e in } \mathds{R}^n,
\end{aligned}
\right.
\end{equation*}
    It is to be expected that the answers should be more complicated because it will be necessary to assume a possible interaction between $\mathbf{W}_{\alpha_1,p_1}\sigma$ and $\mathbf{W}_{\alpha_2,p_2}\sigma$.
\end{enumerate}

\bigskip 

	%

\begin{flushleft}
 {\bf Funding:}  
 E. da Silva acknowledges partial support  from 
	CNPq through grants 140394/2019-2 and
 J. M. do \'O acknowledges partial support from CNPq through grants 312340/2021-4, 409764/2023-0, 443594/2023-6, CAPES MATH AMSUD grant 88887.878894/2023-00
and Para\'iba State Research Foundation (FAPESQ), grant no 3034/2021.  \\
 {\bf Ethical Approval:}  Not applicable.\\
 {\bf Competing interests:}  Not applicable. \\
 {\bf Authors' contributions:}    All authors contributed to the study conception and design. Material preparation, data collection, and analysis were performed by all authors. The authors read and approved the final manuscript.\\
{\bf Availability of data and material:}  Not applicable.\\
{\bf Ethical Approval:}  All data generated or analyzed during this study are included in this article.\\
{\bf Consent to participate:}  All authors consent to participate in this work.\\
{\bf Conflict of interest:} The authors declare that they have no conflict of interest. \\
{\bf Consent for publication:}  All authors consent for publication. \\
\end{flushleft}

\bigskip



\begin{thebibliography}{10}

\bibitem{MR1411441}
D.~R. Adams and L.~I. Hedberg.
\newblock {\em Function spaces and potential theory}, volume 314 of {\em
  Grundlehren der mathematischen Wissenschaften [Fundamental Principles of
  Mathematical Sciences]}.
\newblock Springer-Verlag, Berlin, 1996.

\bibitem{MR1141779}
H.~Brezis and S.~Kamin.
\newblock Sublinear elliptic equations in {${\bf R}^n$}.
\newblock {\em Manuscripta Math.}, 74(1):87--106, 1992.

\bibitem{MR3567503}
D.~Cao and I.~Verbitsky.
\newblock Nonlinear elliptic equations and intrinsic potentials of {W}olff
  type.
\newblock {\em J. Funct. Anal.}, 272(1):112--165, 2017.

\bibitem{MR3556326}
D.~T. Cao and I.~E. Verbitsky.
\newblock Pointwise estimates of {B}rezis-{K}amin type for solutions of
  sublinear elliptic equations.
\newblock {\em Nonlinear Anal.}, 146:1--19, 2016.

\bibitem{MR4530311}
I.~Chlebicka, Y.~Youn, and A.~Zatorska-Goldstein.
\newblock Wolff potentials and measure data vectorial problems with {O}rlicz
  growth.
\newblock {\em Calc. Var. Partial Differential Equations}, 62(2):Paper No. 64,
  41, 2023.

\bibitem{MR3311903}
C.~T. Dat and I.~E. Verbitsky.
\newblock Finite energy solutions of quasilinear elliptic equations with
  sub-natural growth terms.
\newblock {\em Calc. Var. Partial Differential Equations}, 52(3-4):529--546,
  2015.

\bibitem{MR1450953}
G.~Dolzmann, N.~Hungerb\"{u}hler, and S.~M\"{u}ller.
\newblock The {$p$}-harmonic system with measure-valued right hand side.
\newblock {\em Ann. Inst. H. Poincar\'{e} C Anal. Non Lin\'{e}aire},
  14(3):353--364, 1997.

\bibitem{MR727526}
L.~I. Hedberg and T.~H. Wolff.
\newblock Thin sets in nonlinear potential theory.
\newblock {\em Ann. Inst. Fourier (Grenoble)}, 33(4):161--187, 1983.

\bibitem{MR2305115}
J.~Heinonen, T.~Kilpel\"{a}inen, and O.~Martio.
\newblock {\em Nonlinear potential theory of degenerate elliptic equations}.
\newblock Dover Publications, Inc., Mineola, NY, 2006.
\newblock Unabridged republication of the 1993 original.

\bibitem{MR1205885}
T.~Kilpel\"{a}inen and J.~Mal\'{y}.
\newblock Degenerate elliptic equations with measure data and nonlinear
  potentials.
\newblock {\em Ann. Scuola Norm. Sup. Pisa Cl. Sci. (4)}, 19(4):591--613, 1992.

\bibitem{MR1264000}
T.~Kilpel\"{a}inen and J.~Mal\'{y}.
\newblock The {W}iener test and potential estimates for quasilinear elliptic
  equations.
\newblock {\em Acta Math.}, 172(1):137--161, 1994.

\bibitem{MR3174278}
T.~Kuusi and G.~Mingione.
\newblock Guide to nonlinear potential estimates.
\newblock {\em Bull. Math. Sci.}, 4(1):1--82, 2014.

\bibitem{MR3485149}
T.~Kuusi and G.~Mingione.
\newblock Nonlinear potential theory of elliptic systems.
\newblock {\em Nonlinear Anal.}, 138:277--299, 2016.

\bibitem{MR3779689}
T.~Kuusi and G.~Mingione.
\newblock Vectorial nonlinear potential theory.
\newblock {\em J. Eur. Math. Soc. (JEMS)}, 20(4):929--1004, 2018.

\bibitem{MR4043885}
A.~Lischke, G.~Pang, M.~Gulian, and et~al.
\newblock What is the fractional {L}aplacian? {A} comparative review with new
  results.
\newblock {\em J. Comput. Phys.}, 404:109009, 62, 2020.

\bibitem{MR2777530}
V.~Maz'ya.
\newblock {\em Sobolev spaces with applications to elliptic partial
  differential equations}, volume 342 of {\em Grundlehren der mathematischen
  Wissenschaften [Fundamental Principles of Mathematical Sciences]}.
\newblock Springer, Heidelberg, augmented edition, 2011.

\bibitem{MR1890997}
N.~S. Trudinger and X.-J. Wang.
\newblock On the weak continuity of elliptic operators and applications to
  potential theory.
\newblock {\em Amer. J. Math.}, 124(2):369--410, 2002.

\bibitem{MR1747901}
I.~E. Verbitsky.
\newblock Nonlinear potentials and trace inequalities.
\newblock In {\em The {M}az'ya anniversary collection, {V}ol. 2 ({R}ostock,
  1998)}, volume 110 of {\em Oper. Theory Adv. Appl.}, pages 323--343.
  Birkh\"{a}user, Basel, 1999.

\end{thebibliography}



    \end{document}